\documentclass[11pt]{amsart}
\usepackage{amsmath,amssymb,amsfonts,hyperref,enumerate,graphicx,mathtools}
\usepackage[usenames,dvipsnames,svgnames,table]{xcolor} 
\hypersetup{
 colorlinks=true,
 citecolor=Mahogany,
 linkcolor=PineGreen,
 urlcolor=RoyalBlue}
  \usepackage{pagecolor} 
 \usepackage{lipsum}
 \newcommand{\Addresses}{{
  \bigskip
  \footnotesize
  Kenneth~Ward, \textsc{Department of Mathematics and Statistics, American University,
    Washington, DC 20016}\par\nopagebreak
  \textit{E-mail address}: \texttt{kward@american.edu}
  }}
\title{The $q$-Unit Circle}
\author{Kenneth Ward}
\date{\today}
\newtheorem{thm}{Theorem}
\newtheorem*{thm*}{Theorem}
\newtheorem{lem}{Lemma}
\newtheorem{cor}{Corollary}
\newtheorem{prop}{Proposition}
\theoremstyle{definition}
\newtheorem{defn}{Definition}
\newtheorem{rem}{Remark}
\newcommand*{\bfrac}[2]{\genfrac{[}{]}{0pt}{}{#1}{#2}}
\newcommand\T{\rule{0pt}{2.6ex}}
\newcommand\B{\rule[-1.2ex]{0pt}{0pt}}
\allowdisplaybreaks
\begin{document}
\maketitle
\begin{abstract} We define the unit circle for global function fields. We demonstrate that this unit circle (endearingly termed the \emph{$q$-unit circle}, after the finite field $\mathbb{F}_q$ of $q$ elements) enjoys all of the properties akin to the classical unit circle: center, curvature, roots of unity in completions, integrality conditions, embedding into a finite-dimensional vector space over the real line, a partition of the ambient space into concentric circles, M\"{o}bius transformations, a Dirichlet approximation theorem, a reciprocity law, and much more. We extend the exponential action of Carlitz by polynomials to an action by the real line. We show that mutually tangent horoballs solve a Descartes-type relation arising from reciprocity. We define the hyperbolic plane, which we prove is uniquely determined by the $q$-unit circle. We give the associated modular forms and Eisenstein series.
\end{abstract}
\tableofcontents
\section{Introduction}
This paper grew out of a desire to understand what the appropriate analogy to circles in classical hyperbolic geometry would be in global function fields. For example, Descartes' circles \cite{Des}, Soddy spheres \cite{Sod}, the curvature conjectures and results of Graham et al. on Apollonian packings \cite{Gra}, Schmidt arrangements of circles in the plane \cite{Sta}, Conway's topograph \cite{Con}, and so on, depend heavily upon quadratic equations over $\mathbb{Q}$ and the arithmetic structure of quadratic number fields as global fields. These structures all make special use of circles, so it is natural to ask whether we can imagine similar objects over the rational field $\mathbb{F}_q(T)$. One can mimic classical quadratic constructions directly, which does give a glimpse into a structure theory, but it is incomplete: For example, there are types of discriminants which do not appear over $\mathbb{Q}$, special difficulties in characteristic $2$, and most importantly, the classical notion of a circle, while formally definable in extensions of degree $2$ over $\mathbb{F}_q(T)$, lacks a geometric motivation. 

The cyclotomy of global function fields manifests in three distinct constructions: the constant field extensions $\mathbb{F}_{q^d}/\mathbb{F}_q$, the union of Carlitz modules for polynomials in $\mathbb{F}_q[T]$, and a wild part $L_\infty$ at infinity, which together yield the Kronecker-Weber theorem for $\mathbb{F}_q(T)$ (see \cite{Vil2} for an elegant elementary proof of this fact). It is the second component in this list which we study in this paper, which generate the cyclotomic function fields $K_{q,M}$ ($M \in \mathbb{F}_q[T]$) over $\mathbb{F}_q(T)$. Based on the principle that the unit circle derives its structure from the roots of unity in characteristic zero, we aimed to determine whether a similar principle holds over the global fields of characteristic $p > 0$. This paper is the result of that investigation. We construct a ``$q$-unit circle", which we denote by $\mathbb{S}_q$ (inspired by the notation $S^1$ for the classical unit circle), defined over $\mathbb{F}_q(T)$. The structure of the $q$-unit circle $\mathbb{S}_q$ depends only upon the choice of finite field $\mathbb{F}_q$. 

The most important objective of this paper is to show that the definition of $\mathbb{S}_q$ is unequivocally the \emph{correct} one: That is, it should be the only object which exhibits all of the natural properties which one would expect of the unit circle. In particular, embedded in the structure of $\mathbb{S}_q$ should be at least the following: \\

\noindent\emph{Henceforth, ``roots of $q$-unity" are elements of the Carlitz modules for $\mathbb{F}_q[T]$.}

\begin{itemize}
 \item All roots of $q$-unity lie in $\mathbb{S}_q$ and are dense in it;
 \item $\mathbb{S}_q$ should be compact;
 \item The intersection of $\mathbb{S}_q$ with the real line should be a natural finite set;
 \item There should be an action of exponentiation on $\mathbb{S}_q$ by any real number;
 \item An action of ``multiplication" on $\mathbb{S}_q$ by real numbers should fix the center and attain any possible curvature;
 \item $\mathbb{S}_q$ should generate a reciprocity law;
 \item There should be a covering map on $\mathbb{S}_q$ with structure endowed by the roots of $q$-unity;
  \item $\mathbb{S}_q$ should live in a vector space of finite dimension;
  \item The vector space containing $\mathbb{S}_q$ should be the algebraic closure of the real line relative to reciprocity;
 \item Integral elements with conjugates in $\mathbb{S}_q$ should be roots of $q$-unity;
 \item Elements of $\mathbb{S}_q$ which are not roots of $q$-unity should have ``irrational" exponents satisfying a Dirichlet approximation theorem;
 \item $\mathbb{S}_q$ should be nowhere dense and of measure zero in the vector space where it lives;
 \item M\"{o}bius transformations should act on $\mathbb{S}_q$;
 \item An analogue of the Poincar\'{e} disk should be consistent with $\mathbb{S}_q$;
 \item Tangent circles should satisfy relations derived from the reciprocity law of $\mathbb{S}_q$;
 \item Elements of $\mathbb{S}_q$ should allow construction of a normal integral basis in the space where $\mathbb{S}_q$ resides;
  \item The period of the exponential function should belong to the space generated by $\mathbb{S}_q$;
   \item There should be a natural hyperbolic plane, also of finite dimension, associated with $\mathbb{S}_q$; and
   \item There should be modular forms and Eisenstein series which are consistent with the lattice structure in the space generated by $\mathbb{S}_q$.
 \end{itemize}
 This paper proves all of these.
 
 In \S 2, we introduce some notation and list some of the known analogies between number fields and function fields for convenience of the reader. We begin in \S 3.1 with some results demonstrating which roots of $q$-unity are contained in $P$-adic ($P \in \mathbb{F}_q[T]$) completions of $\mathbb{F}_q(T)$, which as it turns out is precisely the same for $\mathbb{S}_q$ as the classical unit circle. We then demonstrate that the circle $\mathbb{S}_q$ is ``perpendicular" to the real line, which here is $(\mathbb{F}_q(T))_\infty$, i.e., the completion of the rational field $\mathbb{F}_q(T)$ at the infinite place $\mathfrak{p}_\infty$ associated with the degree function. In \S 3.2, we give the covering map on $\mathbb{S}_q$, in analogy to classical homotopy, and we prove that $\mathbb{S}_q$ has finite dimension. The M\"{o}bius maps are a bit more delicate in this situation, as the (Carlitz) exponential action is additive for global function fields. In \S 3.3, we complete the Carlitz action on $\mathbb{S}_q$. In \S 3.4, we give the reciprocity laws and the analogue of quadratic forms, demonstrating that $\mathbb{S}_q$ lives in a space which contains all solutions to a form derived from the reciprocity law. We also prove that all conjugates of an element must map into $\mathbb{S}_q$ if it is a root of $q$-unity, in light of the classical result that if the conjugates of a complex number lie on the boundary of the unit disk, then that number is a root of unity. In \S 3.5, we prove that $\mathbb{S}_q$ is nowhere dense, compact, and of measure zero in its vector space $V_q$, as well as a description of the image of $\mathbb{S}_q$ via the reciprocity map. In \S 3.6, we prove the Descartes relation for mutually tangent horoballs on a hyperplane using the form from $\mathbb{S}_q$ introduced in \S 3.4 and give the (perfect) analogy between the Farey-Ford circle packing of the Poincar\'{e} disk \cite[p. 31]{Con} and packings on $\mathbb{P}^1((\mathbb{F}_q(T))_\infty)$ which satisfy the Descartes relation for global function fields. Finally, in \S 3.7, we define the hyperbolic plane $\mathfrak{H}_q$, which is uniquely determined by $q$ and of finite dimension over $(\mathbb{F}_q(T))_\infty$, and give the definitions of its modular forms and Eisenstein series.
 
Having established consistent analogies to both the classical unit circle and the hyperbolic plane, one now has a precise space of finite dimension over the real line in which to think about geodesics, tangency, automorphic forms, isotropy, buildings - and whatever else one wishes!

\section{Notation and basic analogies}
We denote by $S^1$ the complex unit circle $$S^1 = \{ z \in \mathbb{C} \; | \; |z| = 1 \}.$$ In our constructions, it will be helpful to view $S^1$ as the topological completion of the collection $\mathbb{M}$ of all complex roots of unity, i.e., \begin{equation}\label{allrootsofunity} \mathbb{M} = \bigcup_{n \in \mathbb{N}} \{ z \in \mathbb{C} \; | \; z^n = 1 \}.\end{equation} As usual, we let $\mu_n :=\{ z \in \mathbb{C} \; | \; z^n = 1 \},$ denote the set of $n$th roots of unity in $\mathbb{C}$. For a prime integer $p$, let $\mathbb{Q}_p$ denote the field of $p$-adic numbers and $\mathbb{Z}_p$ its integer ring. We also let $q = p^r$ for a positive integer $r$, $\mathbb{F}_q$ the finite field of $q$ elements, and $T$ an indeterminate. Let $P\in\mathbb{F}_q[T]$ be irreducible, and let ${(\mathbb{F}_q(T))}_P$ denote the $P$-adic functions and ${(\mathbb{F}_q[T])}_P$ its integer ring. We denote by $(\mathbb{F}_q(T))_\infty$ the completion of $\mathbb{F}_q(T)$ at the infinite place $\mathfrak{p}_\infty$, which is equal to the pole divisor of the function $T \in \mathbb{F}_q(T)$. Let $\phi_q(u) := u^q$ be the Frobenius and $\mu_T(u) := uT$, multiplication by $T$. For a polynomial $$M = a_n T^n + \cdots + a_1 T +  a_0 \in \mathbb{F}_q[T]$$ the Carlitz exponential action is defined as (see \cite[p. 79]{Hay1} or \cite[Definition 12.2.1]{Vil})  \begin{equation} \label{expaction} M \cdot_q u := u_q^M := a_n (\phi_q + \mu_T)^n(u) + \cdots + a_1 (\phi_q + \mu_T)(u) + a_0 u.\end{equation} (We will see later why the notation $M \cdot_q u$ is useful.) This endows $\overline{\mathbb{F}_q(T)}$ with an $\mathbb{F}_q[T]$-module structure. The set of torsion points over $\mathbb{F}_q$ for the action by a polynomial $M \in \mathbb{F}_q[T]$ is denoted by $$\Lambda_{q,M}: = \{u \in \overline{\mathbb{F}_q(T)} \; | \; u_q^M = 0\}.$$ We let $K_{q,M}:=\mathbb{F}_q(T)(\Lambda_{q,M})$ be the field obtained from $\mathbb{F}_q(T)$ by adjoining $\Lambda_{q,M}$. The field $K_{q,M}$ is called the \emph{cyclotomic function field} for the polynomial $M$. Note that we use the subscript $K_{q,M}$, as this field does depend on $q$ - and that one does not obtain $K_{q^f,M}$ simply by adjoining the finite field $\mathbb{F}_{q^f}$ to $K_{q,M}$. The extension $K_{q,M}/\mathbb{F}_q(T)$ is Galois, with Galois group isomorphic to $(\mathbb{F}_q[T]/M)^*$ via the action defined in \eqref{expaction} \cite[Theorem 2.3]{Hay1}. The set $\Lambda_{q,M}$ is the function field analogue of the classical $n$th roots of unity $\mu_n$, and it is a cyclic additive $\mathbb{F}_q[T]$-module, just as $\mu_n$ is a cyclic group. We denote the collection of all such torsion points (for nonzero $M$) by $$\mathbb{T}_q = \bigcup_{M \in \mathbb{F}_q[T]\backslash\{0\}} \Lambda_{q,M}.$$ This is the analogue for function fields of the collection $\mathbb{M}$ of all roots of unity \eqref{allrootsofunity} . We collect some more of the relevant analogies in Table 1 below. 

\begin{table}[!ht] 
\centering
\small
\caption{
\bf{Some basic analogies between number and function fields.}}
\begin{center}
 \begin{tabular}[t]{||c |c||} 
 \hline
{\sc Number Fields} & {\sc Function Fields}\T\B \\ [0.5ex] 
 \hline\hline
$\mathbb{Q}$ & $\mathbb{F}_q(T)$ \T\B\\ 
 \hline
$n \in \mathbb{Z}$ & $M \in \mathbb{F}_q[T]$\T\B\\ 
 \hline
 $\mu_n$ & $\Lambda_M$\T\B\\
 \hline
 $\zeta^k$, the usual multiplication & $M = \sum_{i=0}^n a_i T^i$,\T \\& $u^M = \sum_{i=0}^n a_i (\phi_q + \mu_T)^i(u).$  \\
 & $\phi_q(u) = u^q$, $\mu_T(u) = uT$\B\\
 \hline
$\zeta^{j + k} = \zeta^j \zeta^k$ &  $u^{M + N} = u^M +u^N$\T\B \\
 \hline
$(\zeta^j)^k = \zeta^{jk}$ & $(u^M)^N = u^{MN}$ \T\B\\ 
 \hline
 $\text{Gal}(\mathbb{Q}(\mu_n)/\mathbb{Q})\simeq (\mathbb{Z}/n\mathbb{Z})^*$ & $\text{Gal}(K_{q,M}/\mathbb{F}_q(T)) \simeq (\mathbb{F}_q[T]/M)^*$\T\B \\ 
 \hline
 N/A & $A_T = \bigcup_M \mathbb{F}_q(T)(\Lambda_M)$\T \\&= maximal abelian / $\mathbb{F}_q(T)$ \\& $+$ tamely ramified at $\infty$ \B \\  \hline $p \in \mathbb{N}$ prime $\Rightarrow$ & $P \in \mathbb{F}_q[T]$ irreducible $\Rightarrow$ \T \\
$\mathbb{Q}(\mu_{p^m})/\mathbb{Q}$ totally ramified at $p$ & $K_{q,P^m}/\mathbb{F}_q(T)$ totally ramified at $P$ \\
Index $p^m(p-1)$ & Index $q^{d(m-1)}(q^d-1)$ ($\deg(P) = d$)\B \\ \hline
 $\mathbb{Q}(\mathbb{M})$ = maximal abelian / $\mathbb{Q}$  &  $A_T A_{1/T} \overline{\mathbb{F}_q}$ maximal abelian / $\mathbb{F}_q(T)$ \T\B\\ \hline
\end{tabular}
\end{center}
\vspace{.5cm}
\end{table}

\section{Structures}
\subsection{Hensel's lemma and roots of $q$-unity} 

Here, we give the function field analogues to statements about $p$-adic roots of unity. We first state Hensel's lemma for $\mathbb{Z}_p$ \cite[Theorem 3.4.1]{Gou}: \begin{lem}[Hensel's lemma] Suppose that $f(x) \in \mathbb{Z}_p[x]$, and that $\alpha \in \mathbb{Z}_p$ is such that $$f(\alpha) \equiv 0 \mod p \qquad \text{ and } \qquad f'(\alpha) \not\equiv 0 \mod p.$$ Then there exists a unique $\beta \in \mathbb{Z}_p$ such that $$f(\beta) = 0 \qquad \text{ and } \qquad \beta \equiv \alpha \mod p.$$ \end{lem} \noindent Via a study of binomial coefficients, one easily obtains the following well-known fact \cite[Theorem 3.1]{Conrad}:
\begin{thm} \label{classicalunity} \begin{enumerate}[(a)] \item If $p=2$, then $\mathbb{M} \cap \mathbb{Q}_p = \mu_2 = \{-1,1\}$. \item If $p$ is odd, then $\mathbb{M} \cap \mathbb{Q}_p = \mu_{p-1}$. \end{enumerate}\end{thm} 

\noindent Of course, the more general version of Hensel's lemma is for complete local rings \cite[Chapter 2.4]{Neu}: 
\begin{lem}[Generalised Hensel's lemma] \label{Henselgeneral} let $R$ be a complete local ring with $\mathfrak{m}$ its maximal ideal, Suppose that $f(x) \in R[x]$, and that $\alpha \in R$ is such that $$f(\alpha) \equiv 0 \mod \mathfrak{m} \qquad \text{ and } \qquad f'(\alpha) \not\equiv 0 \mod \mathfrak{m}.$$ Then there exists a unique $\beta \in R$ such that $$f(\beta) = 0 \qquad \text{ and } \qquad \beta \equiv \alpha \mod \mathfrak{m}.$$
	\end{lem} 
	We now apply this to function fields. In order to do so, let $P \in \mathbb{F}_q[T]$ be monic and irreducible. We have $u_q^P = u \Psi_{q,P}(u)$, where $\Psi_{q,P}(u) \in \in \mathbb{F}_q[T][u]$ may be defined for any $\mathbb{F}_q[T]$-generator $\lambda$ of $\Lambda_{q,P}$ as \cite[Definition 12.3.8]{Vil} \begin{align*} \Psi_{q,P}(u) :&= \prod_{A \in (\mathbb{F}_q[T]/P)^*} (u - \lambda^A) \\&= u^{\Phi(P)} + \beta_{\Phi(P) - 1} u^{\Phi(P) - 1} + \cdots + \beta_1 u + \beta_0. \end{align*} The polynomial $\Psi_{q,P}(u)$ is irreducible over $\mathbb{F}_q[T][u]$. We also have $P \mid \beta_i$ for each $i = 1,\ldots,\Phi(P)-1$ and $\beta_0 = \pm P$. We now determine $\mathbb{T}_q \cap {(\mathbb{F}_q(T))}_P$, i.e., those roots of $q$-unity which lie in the $P$-adic completion of the rational field.

\begin{prop} \label{tq} $\mathbb{T}_q \cap {(\mathbb{F}_q(T))}_P = \Lambda_{q,P-1}$. \end{prop}
\begin{proof} Suppose that $u \in {(\mathbb{F}_q(T))}_P$ satisfies $u_q^A = 0$ for $A \in \mathbb{F}_q[T]$ relatively prime with $P$. By definition of $u_q^A$, it follows that $u$ is integral over $\mathbb{F}_q[T]$ and is thus contained in ${(\mathbb{F}_q[T])}_P$. Let $\deg(A) = d$. We may write $$u_q^A = \sum_{i=0}^d  \bfrac{A}{d}_q u^{q^i},$$ where $\bfrac{A}{0}_q = A$ \cite[Theorem 12.2.5]{Vil}. In particular, with $f(x) = x_q^A \in \mathbb{F}_q[T][x]$, we have $$f(u) = u_q^A = 0 \equiv 0 \mod P \qquad \text{and}\qquad f'(u) = A \not\equiv 0 \mod P.$$ It follows that $u$ is the unique root in ${(\mathbb{F}_q[T])}_P$ of $f(x)$ in the equivalence class of $u$ modulo $P$. If $v \in {(\mathbb{F}_q(T))}_P$ satisfies $v_q^B = 0$ for $B \in \mathbb{F}_q[T]$ relatively prime with $P$, then we may note that $$u_q^{AB} =\left(u_q^{A}\right)_q^B = 0_q^B = 0\qquad \text{and} \qquad v_q^{AB} =\left(v_q^{B}\right)_q^A = 0_q^A = 0.$$ Again by Hensel's lemma and the fact that $f(x) = x_q^{AB}$ satisfies $f'(x) = AB \not\equiv 0 \mod P$, it follows that the equivalence $u = v \mod P$ implies $u = v$. Clearly the polynomial $P - 1 \in \mathbb{F}_q[T]$ is relatively prime with $P$. For all $A \in \mathbb{F}_q[T]$, we have $$A_q^{P-1} = A_q^P - A \equiv A^{q^d} - A \equiv 0 \mod P.$$ The polynomial $f(x) = x_q^{P-1}$ is of degree $q^{\deg(P-1)} = q^{\deg(P)} = q^d$, and $$f(x) \mod P \equiv P - 1 \equiv -1 \not\equiv 0 \mod P,$$ so that $\overline{f}(x) :=f(x) \mod P$ is separable over $\mathbb{F}_q[T]/P$. It follows that $\overline{f}(x)$ factors as $$\overline{f}(x) \equiv \prod_{A \in \mathbb{F}_q[T]/P} (x - A) \mod P.$$ It follows that $f(x) = x_q^{P-1}$ contains a root in every equivalence class modulo $P$ of ${(\mathbb{F}_q[T])}_P$, and by Hensel's lemma, the collection of torsion points in ${(\mathbb{F}_q(T))}_P$ of order relatively prime to $P$ is precisely $\Lambda_{q,P-1}$.

Suppose now that $u \in {(\mathbb{F}_q(T))}_P$ satisfies $u_q^A = 0$ for $A \in \mathbb{F}_q[T]$ divisible by $P$. Let $A = P^k M$ where $(P,M)=1$. We have $$(u_q^M)_q^{P^k} = u_q^{P^k M} = u_q^A= 0.$$ By definition of the exponential action, $u_q^{P^j M} \in {(\mathbb{F}_q[T])}_P$ for each $j=0,\ldots,k-1$. Letting $\alpha := u_q^{P^{k-1} M}$, we then have $\alpha_q^{P} = 0$. The polynomial $f(x) = x_q^P$, for which $f'(x) = \bfrac{P}{0}_q=P$, then satisfies $$f(\alpha) = 0\qquad \text{and} \qquad v_P(f'(\alpha)) = v_P(P) = 1.$$ By definition, $$f(\alpha) =  \alpha_q^P = \alpha^{\Phi(P)+1} + \beta_{\Phi(P) - 1} \alpha^{\Phi(P)} + \cdots + \beta_1 \alpha^2 + \beta_0 \alpha.$$ As $P \mid \beta_0,\beta_1,\ldots,\beta_{\Phi(P) - 1}$, we easily obtain $P \mid \alpha$. Letting $\alpha = P N$, we obtain $$0 = \alpha^{\Phi(P)+1} + \beta_{\Phi(P) - 1} \alpha^{\Phi(P)} + \cdots + \beta_1 \alpha^2 + \beta_0 \alpha \equiv \beta_0 PN \mod P^3.$$ As $\beta_0 = \pm P$, it follows that $P\mid N$, and hence that $P^2 \mid \alpha$. It follows by Hensel's lemma once more that $\alpha = 0$. As $0 \in \Lambda_{q,P-1}$, this concludes the proof. \end{proof} 

It is also possible to formulate and prove a version of this result at infinity. One expects that the intersection at infinity will be quite small and will not contain the constant units $\mathbb{F}_q^*$, as the $\mathbb{F}_q[T]$ action is additive. This is completely true, and in fact, it contains \emph{no} units:

\begin{prop} \label{infinitytq} $\mathbb{T}_q \cap (\mathbb{F}_q(T))_\infty = \{0\}$. \end{prop} 

\begin{proof} Let $$u =  \sum_{m=-\infty}^n a_m T^m \in (\mathbb{F}_q(T))_\infty.\qquad (a_m \in \mathbb{F}_q)$$ Suppose that $u \neq 0$, and that $u$ satisfies $u_q^M = 0$ for some $M \in \mathbb{F}_q[T]\backslash\{0\}$. With $\deg(M) = d$, we have $$u_q^M = \sum_{i=0}^d \bfrac{M}{i}_q u^{q^i}.$$ Thus $n = 0$, so that $a_0 \neq 0$ and $a_0^M = 0$, so that $a_0 \in \Lambda_{q,M}$. As $\Lambda_{q,M}$ is an $\mathbb{F}_q$-vector space, it follows that $\mathbb{F}_q \subset \Lambda_{q,M}$. For each $i=0,\ldots,d$, $\bfrac{M}{i}_q$ is a polynomial in $\mathbb{F}_q[T]$ of degree $(d-i)q^i$ \cite[Proposition 1.1]{Hay1} , which assumes a maximum value among $i=0,\ldots,d$ at $i=d-1$, where it is equal to $q^{d-1}$.  As $\deg(a_0) = 0$, it follows that $$\deg(a_0^M) = \deg\left(\sum_{i=0}^d \bfrac{M}{i}_q a_0^{q^i}\right)  = \max_i \deg\left(\bfrac{M}{i}_q\right) = \deg\left(\bfrac{M}{d-1}_q\right) = q^{d-1}.$$ In particular, $-\infty = \deg(0) = \deg(a_0^M) = q^{d-1}$, a contradiction. \end{proof} 

\begin{rem} Proposition \ref{infinitytq} is similar to the statement that $\mathbb{M} \cap \mathbb{R} = \{-1,1\}$, and this reflects the fact that the Carlitz torsion modules $\Lambda_{q,M}$ are additive modules with $0$ as the unique ``unit", whereas the groups $\mu_n$ are multiplicative.\end{rem}

\subsection{The $q$-unit circle $\mathbb{S}_q$: Center, curvature, dimension}

We must now proceed to define the $q$-unit circle $\mathbb{S}_q$, and along with that, the notions of center and curvature which should be attached to such an object.

\begin{defn}[The $q$-unit circle] We denote by $\mathbb{S}_q$ the completion of $\mathbb{T}_q$ at a fixed choice of branch $\mathfrak{P}$ of infinity. $\mathbb{S}_q$ is called the \emph{$q$-unit circle}. \end{defn} 

The collection $\mathbb{T}_q$ forms an additive group: If $u,v \in \mathbb{T}_q$, then there exist $M,N \in \mathbb{F}_q[T]$ such that $u_q^M = 0$ and $v_q^N = 0$. Thus $$(u+v)_q^{MN} = (u_q^M)_q^N + (v_q^N)_q^M = 0_q^N + 0_q^M = 0,$$ so that $u+v \in \Lambda_{q,MN} \subset \mathbb{T}_q$. As $\mathbb{F}_q \cdot \Lambda_{q,M} = \Lambda_{q,M}$ for each $M \in \mathbb{F}_q[T]$, it follows that $\mathbb{F}_q \cdot \mathbb{T}_q = \mathbb{T}_q$. Note that multiplication by $\mathbb{F}_q$ agrees with the Carlitz action by elements of $\mathbb{F}_q$ \cite[Remark 12.2.4]{Vil}. Finally, $M(\mathbb{T}_q) \subset \mathbb{T}_q$ for all $M \in \mathbb{F}_q[T]$, as $u \in \mathbb{T}_q$ satisfies $u_q^N = 0$ for some $N \in \mathbb{F}_q[T]$, so that $u_q^M$ satisfies $$(u_q^M)_q^N = u_q^{MN} = u_q^{NM} = (u_q^N)_q^M = 0_q^M = 0.$$ By passing to the completion, we thus obtain:
\begin{prop} \label{sqvectorspace} $\mathbb{S}_q$ admits the structure of an additive group, $\mathbb{F}_q[T]$-module, and $\mathbb{F}_q$-vector space. \end{prop}

We are now ready to define a notion of center for the $q$-unit circle $\mathbb{S}_q$:
\begin{defn}[Center] The \emph{center} of $\mathbb{S}_q$ is the formal limit $\widetilde{\infty} = \lim_{n\rightarrow \infty} T^n$. \end{defn}
The center of $\mathbb{S}_q$ may be viewed as invariant under addition by elements of $(\overline{\mathbb{F}_q(T)})_\infty$. In the same way that the curvature of the classical unit circle $S^1$ is equal to 1, we say via the additive $\mathbb{F}_q[T]$-action $M \cdot_q u = u_q^M$ that the curvature of $\mathbb{S}_q$ is equal to $0$. More generally, we may also define:
\begin{defn}[Curvature]\label{def:curvature} For each $u \in (\mathbb{F}_q(T))_\infty$, let $\mathbb{S}_q(u):= u + \mathbb{S}_q$. The \emph{curvature} $k(u)$ of $\mathbb{S}_q(u)$ is defined by $k(u):=u$. \end{defn} 

We note that this is consistent with the structure of $\mathbb{S}_q$ as an $\mathbb{F}_q$-vector space, by Proposition \ref{infinitytq}. Over the complex plane $\mathbb{C}$, the curvature of a circle is a real number as well: The curvature $k(r)$ of a circle $C_r$ of radius $r$ is equal to $k(r) = 1/r$, and thus
\begin{equation} \label{ccurvature} \lim_{r \rightarrow 0^+} k(r) = \infty.\end{equation} 
If $C_r$ is centered at the origin, \eqref{ccurvature} is equivalent to the fact that the set of points on $C_r$ converge uniformly to the origin. For $\mathbb{S}_q(u)$, with $u = T^n$, we obtain in the valuation at $\mathfrak{p}_\infty$ that $$\lim_{n \rightarrow \infty} k(T^n) = \lim_{n \rightarrow \infty} T^n = \widetilde{\infty},$$ and hence that $\mathbb{S}_q(T^n)$ converges in a natural way to the center of $\mathbb{S}_q$ as $n \rightarrow \infty$. 

Just as multiplication by $n \in \mathbb{Z}$ generates a covering map of the classical unit circle $S^1$, we find a similar map for function fields on the $q$-unit circle:

\begin{prop} \label{covering} The map $M \cdot_q u = u_q^M$ induces a covering $M: \mathbb{S}_q \rightarrow \mathbb{S}_q$. 	
\end{prop}

\begin{proof} If $u \in \mathbb{T}_q$, then by definition of $\mathbb{T}_q$, $u \in \Lambda_{q,N}$ for some $N \in \mathbb{F}_q[T]$. Hence $N \cdot_q u = u_q^N = 0$. In particular, we obtain $$N \cdot_q (M \cdot_q u) = NM \cdot_q u = MN \cdot_q u = M \cdot_q (N \cdot_q u) = M \cdot_q 0 = 0,$$ and hence that $M \cdot_q u \in \Lambda_{q,N} \subset \mathbb{T}_q$, so that the action of $M$ induces a map $M  : \mathbb{T}_q \rightarrow \mathbb{T}_q$. To see that this is in fact a covering, let once more $u \in \Lambda_{q,N}$ for some $N \in \mathbb{F}_q[T]$ once . We wish to find a $v \in \mathbb{T}_q$ such that $M \cdot_q v = u$. Let $$N = \prod_{i=1}^r P_i^{\alpha_i}$$ denote the decomposition of $N$ into distinct primes $P_1,\ldots,P_r \in \mathbb{F}_q[T]$. We also have the prime decomposition of torsion points \cite[Proposition 12.2.13]{Vil}  \begin{equation}\label{primedecomp} \Lambda_{q,N} = \bigoplus_{i=1}^r \Lambda_{q,P_i^{\alpha_i}}.\end{equation} For each irreducible polynomial $P \in \mathbb{F}_q[T]$ and $n \in \mathbb{N}$, the homomorphism of $\mathbb{F}_q[T]$ modules \begin{equation} \label{Psurj} P \cdot_q u : \Lambda_{q,P^{n+1}} \rightarrow \Lambda_{q,P^n}\end{equation} is surjective \cite[Proposition 1.5]{Hay1}. We therefore let $$M = \prod_{j=1}^s Q_i^{\beta_j}$$ be the decomposition of $M$ into distinct primes $Q_1,\ldots,Q_s \in \mathbb{F}_q[T]$. Then $MN$ admits the factorisation $$MN = \prod_{i=1}^r P_i^{\alpha_i + \beta_i} \prod_{Q_j \neq P_i} Q_j^{\beta_j}.$$ By \eqref{Psurj}, the maps $$P_i^{\beta_i} \cdot_q u: \Lambda_{q,P_i^{\alpha_i + \beta_i }} \rightarrow \Lambda_{q,P_i^{\alpha_i}} \qquad  (i=1,\ldots,r)$$ and $$Q_j^{\beta_j} \cdot_q u: \Lambda_{q,Q_j^{\beta_j}} \rightarrow \{0\} \qquad (Q_j \neq P_i)$$ are surjective. Also, similarly to \eqref{primedecomp}, we find \begin{equation}\label{secondprimedecomp} \Lambda_{q,MN} = \left(\bigoplus_{i=1}^r \Lambda_{q,P_i^{\alpha_i + \beta_i }}\right) \bigoplus \left(\bigoplus_{Q_j \neq P_i} \Lambda_{q,Q_j^{\beta_j}} \right).\end{equation} Therefore, the map $$M \cdot_q u : \Lambda_{q,MN} \rightarrow \Lambda_{q,N}$$ is surjective. As $\Lambda_{q,MN} \subset \mathbb{T}_q$, the induced map $M  : \mathbb{T}_q \rightarrow \mathbb{T}_q$ is then also surjective. By passing to completions, the result follows. 
	
\end{proof}

\begin{rem} Much like the exponentiation $z \rightarrow z^n$ on $\mathbb{C}$, the extension of $M: \mathbb{S}_q \rightarrow \mathbb{S}_q$ to a map on the completion $(\overline{\mathbb{F}_q(T)})_\mathfrak{P}$ of $\overline{\mathbb{F}_q(T)}$ at infinity is surjective, as we may write $$M \cdot_q u = \sum_{i=1}^d \bfrac{M}{i}_q u^{q^i},\qquad \bfrac{M}{i}_q \in \mathbb{F}_q[T],\quad  i=1,\ldots,d=\deg(M),$$ \cite[Proposition 1.1]{Hay1}, so that the equation $M \cdot_q u = \alpha$ is simply an algebraic equation over $(\overline{\mathbb{F}_q(T)})_\mathfrak{P}$ for $\alpha \in (\overline{\mathbb{F}_q(T)})_\infty$, and hence takes its roots in $(\overline{\mathbb{F}_q(T)})_\mathfrak{P}$.
\end{rem}

By definition, the complex unit circle $S^1$ exists within $\mathbb{C}$, a vector space of dimension 2 over $\mathbb{R}$. We now proceed to show that $\mathbb{S}_q$ similarly lies within a vector space of finite dimension over $(\mathbb{F}_q(T))_\infty$. For this, we need a few lemmas. We again let $\mathfrak{p}_\infty$ denote the place of $\mathbb{F}_q(T)$ at infinity. For an algebraic extension $K/\mathbb{F}_q(T)$ and place $\mathfrak{P}$ of $K$ above $\mathfrak{p}_\infty$, we also let $e(\mathfrak{P}|\mathfrak{p}_\infty)$ denote the ramification index of $\mathfrak{P}|\mathfrak{p}_\infty$ and $f(\mathfrak{P}|\mathfrak{p}_\infty)$ the inertia degree.

\begin{lem} \label{Pninf} Let $M = P^n$, where $P^n \in \mathbb{F}_q[T]$ is irreducible. Let $\mathfrak{P}$ be a place of $K_{q,P^n}$ above $\mathfrak{p}_\infty$. Then 
$e(\mathfrak{P}|\mathfrak{p}_\infty) = q-1$ and $f(\mathfrak{P}|\mathfrak{p}_\infty) = 1$.
\end{lem}

\begin{proof} \cite[Theorem 3.2]{Hay1}. \end{proof}

The proof of Lemma \ref{Pninf} is a detailed study of the Newton polygon at infinity for the Carlitz action $M \cdot_q u$. 

\begin{lem} Let $M \in \mathbb{F}_q[T]$. Let $\mathfrak{P}$ be a place of $K_{q,M}$ above $\mathfrak{p}_\infty$. Then 
$e(\mathfrak{P}|\mathfrak{p}_\infty) = q-1$ and $f(\mathfrak{P}|\mathfrak{p}_\infty) = 1$.
\end{lem}

\begin{proof} Let $$M = \prod_{i=1}^r P_i^{\alpha_i}$$ be the decomposition of $M$ into distinct prime factors $P_1,\ldots,P_r \in \mathbb{F}_q[T]$. By Lemma \ref{Pninf}, the ramification indices at infinity within $K_{q,P_i^{\alpha_i}}$ $(i=1,\ldots,r)$ are equal to $q-1$, i.e., $\mathfrak{p}_\infty$ is tamely ramified in $K_{q,P_i^{\alpha_i}}$. By Abhyankar's lemma and the fact that $$K_{q,M} = \prod_{i=1}^r  K_{q,P_i^{\alpha_i}},$$ it follows that $e(\mathfrak{P}|\mathfrak{p}_\infty) = q-1$. 

We must now show that $f(\mathfrak{P}|\mathfrak{p}_\infty) = 1$. This is a consequence of the fact that the completion $(K_{q,M})_\mathfrak{P}$ of $K_{q,M}$ at $\mathfrak{P}$ satisfies $$(K_{q,M})_\mathfrak{P} = \mathbb{F}_q((^{q-1} \sqrt{-1/T}))$$ \cite[Lemma 1.1]{GaRo}.   \end{proof}

We have shown that $$[(\mathbb{F}_q(T))_\infty(\Lambda_{q,M}):(\mathbb{F}_q(T))_\infty] = e(\mathfrak{P}|\mathfrak{p}_\infty)f(\mathfrak{P}|\mathfrak{p}_\infty) = q-1$$ for every $M \in \mathbb{F}_q[T]$. This allows us to prove the following:

\begin{thm} \label{vsq-1} The field $(\mathbb{F}_q(T))_\infty (\mathbb{S}_q)$ forms a vector space over $(\mathbb{F}_q(T))_\infty $ of dimension equal to $q-1$.
\end{thm}

\begin{proof} Let $\mathbf{P} = \{P_1,P_2,\ldots\} \subset \mathbb{F}_q[T]$ denote the collection of monic, irreducible polynomials in $\mathbb{F}_q[T]$. We define the sequence of infinite tuples $\{\mathbf{a}_1,\mathbf{a}_2,\ldots\}$, where for each $n \in \mathbb{N}$, $$\mathbf{a}_n := (n,n-1,\ldots,1,0,0,\ldots) \quad \text{and} \quad\mathbf{P}^{\mathbf{a}_n} := P_1^n P_2^{n-1} \cdots P_n^1.$$ Then the sequence of polynomials $\{\mathbf{P}^{\mathbf{a}_n}\}_{n \in \mathbb{N}} \subset \mathbb{F}_q[T]$ satisfies the following conditions:

\begin{enumerate}
\item If $M \in \mathbb{F}_q[T]$, then there exists an $n \in \mathbb{N}$ such that $M\mid \mathbf{P}^{\mathbf{a}_n}$.
\item For each $n \in \mathbb{N}$, $\Lambda_{q,\mathbf{P}^{\mathbf{a}_n}} \subset \Lambda_{q,\mathbf{P}^{\mathbf{a}_{n+1}}}$.
\item For each $n \in \mathbb{N}$, $[(\mathbb{F}_q(T))_\infty(\Lambda_{q,\mathbf{P}^{\mathbf{a}_n}}):(\mathbb{F}_q(T))_\infty] = q-1$.
\end{enumerate}
It follows that $$\mathbb{T}_q \subset V: = \bigcup_{n \in \mathbb{N}} (\mathbb{F}_q(T))_\infty(\Lambda_{q,\mathbf{P}^{\mathbf{a}_n}}),$$ and that $V$ is a vector space which satisfies $\dim(V/(\mathbb{F}_q(T))_\infty) = q-1$. By completeness of $(\mathbb{F}_q(T))_\infty$, it then follows that $V$ is also complete, and hence that $\mathbb{S}_q \subset V$. Therefore, $V = (\mathbb{F}_q(T))_\infty (\mathbb{S}_q)$. \end{proof}
\begin{defn} $V_q := (\mathbb{F}_q(T))_\infty (\mathbb{S}_q)$. \end{defn} 

In order to determine curvature, we have already defined an additive action on $\mathbb{S}_q$ by the subspace $\mathbb{F}_q(T))_\infty \subset V_q$ (Definition~\ref{def:curvature}). We may now extend this to all of $V_q$: 

\begin{defn} \label{def:curvaturecenter} If $v \in V_q$, we let $\mathbb{S}_q(v) = v + \mathbb{S}_q$.
\end{defn}

\begin{rem} In Definition \ref{def:curvaturecenter} the element $v$ is the function field analogue to the curvature-center of a general circle $C \subset \mathbb{C}$, which is defined as $k_C \cdot z_C$, where $k_C = 1/r_C$ is the usual definition of curvature of $C$, equal to the inverse of the radius $r_C$ of $C$, and $z_C$ is the center of $C$.	
\end{rem}

We may now define the M\"{o}bius transformations. If $v \in V_q$ and $w \in \mathbb{S}_q(v) \cap \mathbb{S}_q$, then there exists $s \in \mathbb{S}_q$ such that $w = v + s$. As $\mathbb{S}_q$ is an additive group, it follows that $v = w - s \in \mathbb{S}_q$, and hence that $\mathbb{S}_q(v) = \mathbb{S}_q$. Furthermore, as $0 \in \mathbb{S}_q$, it follows that $v \in \mathbb{S}_q(v)$ for each $v \in V_q$. Thus, the sets $\mathbb{S}_q(v)$ ($v \in V_q$) partition $V_q$. We let $R_q$ denote a collection of representatives for this partition, so that \begin{equation} \label{disjoint} V_q = \bigsqcup_{\alpha \in R_q} \mathbb{S}_q(\alpha).\end{equation} Hence for each $w \in V_q$, the action $w_+ : V_q \rightarrow V_q$ defined by $w_+(v) := w + v$ induces a permutation on the collection $\{\mathbb{S}_q(\alpha)\}_{\alpha \in R_q}$. 

\begin{lem} \label{infinitysq} $\mathbb{S}_q \cap (\mathbb{F}_q(T))_\infty = \{0\}$. 
\end{lem}

\begin{proof} Suppose that $\lambda \in \mathbb{S}_q \cap (\mathbb{F}_q(T))_\infty$ is nonzero. Let $\{\lambda_n\}_{n \in \mathbb{N}} \subset \mathbb{T}_q \backslash \{0\}$ such that $\lim_{n \rightarrow \infty} \lambda_n  = \lambda$. We let $\mathfrak{P}$ denote the place of $V_q$ above $\mathfrak{p}_\infty$. By \cite[Lemma 1.5]{GaRo}, $(v_\mathfrak{P}(\lambda_n),q-1) = 1$ for all $n \in \mathbb{N}$. By the non-Archimedean property, there exists $N \in \mathbb{N}$ such that $v_\mathfrak{P}(\lambda_n) = v_\mathfrak{P}(\lambda)$ for all $n \geq N$, from which it follows that \begin{equation} \label{lambdastarrelprime} (v_\mathfrak{P}(\lambda),q-1) = 1. \end{equation} But as $\lambda \in (\mathbb{F}_q(T))_\infty \backslash \{0\}$ and $\mathfrak{p}_\infty$ is fully ramified in $V_q$, we find $v_\mathfrak{P}(\lambda) = (q-1)v_{\mathfrak{p}_\infty}(\lambda)$, whence $q-1$ divides $v_\mathfrak{P}(\lambda)$, contradicting \eqref{lambdastarrelprime}.
\end{proof}

We note that when the element $\lambda$ in the proof of Lemma \ref{infinitysq} lies in $\mathbb{T}_q$, we obtain a second proof of Proposition \ref{infinitytq}. Via Lemma \ref{infinitysq}, each element of $(\mathbb{F}_q(T))_\infty$ can be chosen to represent distinct sets in the partition \eqref{disjoint}, so we may view $(\mathbb{F}_q(T))_\infty$ as a subset of $R_q$.

\begin{defn}\label{def:mobiusadditive}
The action of $V_q$ on itself defined for each $w \in V_q$ by $w_+ : V_q \rightarrow V_q$ forms the group of \emph{additive M\"{o}bius transformations} of $\mathbb{S}_q$.
\end{defn}
The additive M\"{o}bius transformations of $\mathbb{S}_q$, which are consistent with the additive Carlitz action, are the analogy to the classical \emph{diagonal} M\"{o}bius transformations. Clearly, by Proposition \ref{covering} and additivity of the map $M \cdot_q$, we have for each $v \in V_q$ that  $$M \cdot_q {\mathbb{S}_q(v)} = M \cdot_q (v + \mathbb{S}_q) =  M \cdot_q v + M \cdot_q \mathbb{S}_q = M \cdot_q v +  \mathbb{S}_q =  \mathbb{S}_q(M \cdot_q v).$$ Furthermore, as $\mathbb{S}_q$ is an additive group, we obtain for each $v,w \in V_q$ that $$\mathbb{S}_q(v+w) = v + w + \mathbb{S}_q = (v +  \mathbb{S}_q) + (w + \mathbb{S}_q) = \mathbb{S}_q(v) + \mathbb{S}_q(w).$$ We have thus obtained:

\begin{prop}
The set-valued map $\mathbb{S}_q : 	V_q \rightarrow V_q$ defined for each $v \in V_q$ by $\mathbb{S}_q \circ v = \mathbb{S}_q(v)$ is $\mathbb{F}_q[T]$-linear via the Carlitz action.
\end{prop}

This is a bit subtle and represents a key difference with the classical unit circle, because while $V_q$ is a vector space in the usual sense, the additive M\"{o}bius transformations come from the $\mathbb{F}_q[T]$-module structure of $\mathbb{S}_q$. The group $GL(V_q)$ also acts on $V_q$ via the usual matrix multiplication, and we may thus also define:

\begin{defn} \label{def:mobius} The action of $GL(V_q)$ on $V_q$ forms the group of \emph{multiplicative M\"{o}bius transformations} of $\mathbb{S}_q$.
\end{defn}

Having defined the two types of M\"{o}bius transformations, we notice that something is missing: The Carlitz action is still an $\mathbb{F}_q[T]$-action, whereas arithmetic on completions necessitates that it be an $(\mathbb{F}_q(T))_\infty$-action. This is the objective of the next section.

\subsection{Completing the Carlitz action}

The Carlitz exponential map defines an action on $\overline{\mathbb{F}_q(T)}$ by polynomials in $\mathbb{F}_q[T]$. We will show that it is possible to extend this in a natural way to an action by the Laurent series in $(\mathbb{F}_q(T))_\infty$. Let $\mathfrak{P}$ be the infinite place of $V_q$ above $\mathfrak{p}_\infty$. We state a few preliminary results. The first places a hard lower bound on the order of points on the $q$-unit circle at infinity: 

\begin{lem}\label{val-1} If $u \in \mathbb{S}_q$, then $v_\mathfrak{P}(u) \geq -1$. 
	\end{lem}
	\begin{proof} Let $\{u_n\}_{n \in \mathbb{N}} \subset \mathbb{T}_q$ be chosen such that $\lim_{n \rightarrow \infty} u_n = u$. By \cite[Lemma 1.5]{GaRo}, $v_\mathfrak{P}(u_n) \geq -1$ for all $n \in \mathbb{N}$. If $N \in \mathbb{N}$ is such that $v_\mathfrak{P}(u - u_n) \geq 0$ for all $n \geq N$, then we obtain $$v_\mathfrak{P}(u) = v_\mathfrak{P}(u-u_n + u_n ) \geq \min\{v_\mathfrak{P}(u-u_n),v_\mathfrak{P}(u_n )\} \geq -1.$$ Hence the result. \end{proof}
	
	Of course, if $v_\mathfrak{P}(\lambda) \geq 0$, then by the non-Archimedean property and $v_{\mathfrak{P}}(T) = -(q-1)$, we have $$v_\mathfrak{P}(\lambda_q^T) = v_\mathfrak{P}(\lambda^q + T\lambda) = \min\{v_\mathfrak{P}(\lambda^q),v_\mathfrak{P}( T\lambda)  \} = v_\mathfrak{P}( T\lambda) = v_\mathfrak{P}(\lambda) - (q-1).$$ Let $m \in \mathbb{N}$. The set $\Lambda_{q,T^m-1}$ is equal to the fixed point set of the action $(\phi_q + \mu_T)^m(u)$, consisting of precisely those $\lambda$ which satisfy \begin{equation}\label{fixedpoint}\lambda_q^{T^m-1} =  (\phi_q + \mu_T)^m(\lambda) - \lambda = 0.\end{equation} Let $\lambda = \lambda_0 \in \Lambda_{q,T^m-1}\backslash \{0\}$, and consider a series of the form \begin{equation}\label{laurentseries}\sum_{k=1}^{\infty} a_{-k} \lambda_{-k}.\quad \qquad \lambda_{-k+1} \in (\phi_q + \mu_T)(\lambda_{-k})\qquad (k \in \mathbb{N})\end{equation} If we choose $\lambda_{-lm} = \lambda_0 \neq 0$ for each $l= 1,2,\ldots$, then as as $u_0 \neq 0$, the series \eqref{laurentseries} does not converge. In particular, the converse of \cite[Lemma 1.6]{GaRo} is false: It very well might occur that $v_\mathfrak{P}(\lambda_q^T) > 0$ but $v_\mathfrak{P}(\lambda) = -1$. To this end, we prove a minor result about periodicity.
	
	\begin{lem}\label{infinityperiodic} Let $m \in \mathbb{N}$. There exists a generator $\lambda$ of $\Lambda_{q,T^m-1}$ such that the sequence of valuations $v_\mathfrak{P}(\lambda), v_\mathfrak{P}(\lambda_q^T),v_\mathfrak{P}(\lambda_q^{T^2})\ldots$ is periodic with period $m$.
	\end{lem}
	
	\begin{proof} By \cite[Proposition 1.10]{GaRo}, there exists $\lambda \in \Lambda_{q,T^m-1}$ such that $v_\mathfrak{P}(\lambda) = (m-1)(q-1) -1$, and by \cite[Lemma 1.5]{GaRo}, $\lambda$ is a generator of $\Lambda_{q,T^m-1}$. By \cite[Lemma 1.6]{GaRo}, The elements $\lambda, \lambda_q^T, \ldots, \lambda_q^{T^{m-1}}$ have respective valuations $j(q-1)-1$ where $j=m-1,m-2,\ldots,0$. As $\lambda \in \Lambda_{q,T^m-1}$, we have $\lambda^{T^m} = \lambda$, hence the result.
	\end{proof}

In order to define the completed Carlitz action, we first require convergent sums of successive applications of the inverse of the map $u \rightarrow u_q^T$ to elements of $\mathbb{T}_q$. In the next result, we show that one can always find such a convergence. The proof is a bit delicate, as the inverse of this exponentiation by $T$ does not always increase valuations at $\mathfrak{P}|\mathfrak{p}_\infty$. 

\begin{prop}\label{carlitzconverge}
Let $u = v_0\in \mathbb{T}_q$, and let $M \in \mathbb{F}_q[T]$ be such that $u$ is a generator of $\Lambda_{q,M}$. One may choose a sequence $\{v_{-k}\}_{k\in \mathbb{N}}$ such that: 
\begin{enumerate} \item $v_{-k} \in \Lambda_{q,T^{k} M}$ for each $k \in \mathbb{N}$, \item $v_{-k+1} \in (\phi_q + \mu_T)(v_{-k})$ for each $k \in \mathbb{N}$, and \item  the series \begin{equation}\label{seriess}\mathfrak{S} = \sum_{k=1}^{\infty} a_{-k} v_{-k}\qquad a_{-k} \in \mathbb{F}_q \qquad (k \in \mathbb{N})\end{equation} is convergent at the place $\mathfrak{P}|\mathfrak{p}_\infty$ of $V_q$. \end{enumerate} Furthermore, each $v_{-k}$ is a generator of $\Lambda_{q,T^{k} M}$, and the sequence of valuations $\{v_\mathfrak{P}(v_{-k})\}_{k \in \mathbb{N}}$ is eventually strictly increasing with a slope of $q-1$ for all $k \geq K$, where an upper bound for $K$ may be determined by only $q$ and $M$.
\end{prop} 

\begin{proof} By \cite[Proposition 1.10]{GaRo}, there exists $u_{-k} \in \Lambda_{q,T^{k} M}$ such that \begin{equation} \label{hival} v_\mathfrak{P}(u_{-k}) = (k + \deg(M) - 1)(q-1)-1.\end{equation} Hence the series $\mathfrak{S}$ converges. By \cite[Lemma 1.5]{GaRo}, we obtain for each $k$ that $u_{-k} \notin \Lambda_{q,T^{k-1} M}$. Let $M = T^a N$, where $(N,T) = 1$. Then $T^k M = T^{k+a} N$. By \eqref{primedecomp}, we may write $u_{-k} = u_{-k,T^{k+a}} + u_{-k,N}$ uniquely with $u_{-k,T^{k+a}} \in \Lambda_{q,T^{k+a}} \backslash \Lambda_{q,T^{k+a-1}}$ and $u_{-k,N} \in \Lambda_{q,N}$. Clearly $u_{-k,T^{k+a}}$ is a generator of $\Lambda_{q,T^{k+a}}$. As in \cite[Lemma 1.6]{GaRo}, we obtain $$v_\mathfrak{P}\left({(u_{-k})}_q^{T^{k+a}} \right) =(\left[k + \deg(M) - 1\right] - (k+a))(q-1)-1 = (\deg(N) - 1)(q-1) -1$$ which in turn by \cite[Lemma 1.5]{GaRo} implies that ${(u_{-k})}_q^{T^{k+a}}$ is a generator of $\Lambda_{q,N}$. But \begin{align*} {(u_{-k})}_q^{T^{k+a}}&={( u_{-k,T^{k+a}} + u_{-k,N})}_q^{T^{k+a}} \\&={( u_{-k,T^{k+a}})}_q^{T^{k+a}} + {(u_{-k,N})}_q^{T^{k+a}} \\&= {(u_{-k,N})}_q^{T^{k+a}},\end{align*} which by \cite[Proposition 12.2.21]{Vil} is also a generator of $\Lambda_{q,N}$. Hence with $A \in \mathbb{F}_q[T]$ such that $A T^{k+a} \equiv 1 \pmod{N}$, we obtain again by \cite[Proposition 12.2.21]{Vil} that $$u_{-k,N} = {(u_{-k,N})}_q^{A T^{k+a}} = {({(u_{-k,N})}_q^{T^{k+a}})}_q^A$$ is a generator of $\Lambda_{q,N}$. It follows that $u_{-k}$ is a generator of $\Lambda_{q,T^k M}$, as desired.

For each $k \in \mathbb{N}$, we let $u_{-k}$ be such an element of $\Lambda_{q,T^k M}$. Thus, for each $0 \leq j \leq k$, the element ${(u_{-k})}_q^{T^j}$ is a generator of $\Lambda_{q,T^{k-j} M}$. As $\Lambda_{q,M}$ is finite, we may pass to a subsequence of $\{u_{-k}\}_{k \in \mathbb{N}}$ for which the generator ${(u_{-k})}_q^{T^k}$ of $\Lambda_{q,M}$ remains constant. From that, we may select a subsequence for which the generators ${(u_{-k})}_q^{T^k}$ of $\Lambda_{q,M}$ and ${(u_{-k})}_q^{T^{k-1}}$ of $\Lambda_{q,TM}$ remain constant, and so on inductively, which yields a subequence of generators $\{u_{-k_n}\}_{n \in \mathbb{N}}$ for which $v_\mathfrak{P}(u_{-k_n} - u_{-k_{n+1}}) \geq n(q-1)-1$. Taking $u_{-k}^* =  {(u_{-k_n})}_q^{T^{k_n-k}}$, one easily sees that $u_{-k}^*$ is a generator of $\Lambda_{q,T^k M}$ which satisfies \eqref{hival}, that the series $\mathfrak{S}$ \eqref{seriess} is convergent for the sequence $\{u_{-k}^*\}_{k \in \mathbb{N}}$ and that $u_{-k+1}^* \in (\phi_q + \mu_T)(u_{-k}^*)$ for each $k \in \mathbb{N}$.

By definition of this sequence, $(u_{-k}^*)_q^{T^k} = w$ for all $k \in \mathbb{N}$, where $w$ is a generator of $\Lambda_{q,M}$ which might not be equal to $u$. By \cite[Proposition 12.2.21]{Vil}, there exists $B \in \mathbb{F}_q[T]$ with $(T,B) = 1$ such that ${(w)}_q^B = u$. It follows that one also has $${((u_{-k}^*)_q^{B})}_q^{T^k} = (u_{-k}^*)_q^{T^kB} = {(w)}_q^B = u$$ for each $k \in \mathbb{N}$, that $(u_{-k}^*)_q^{B}$ is a generator of $\Lambda_{q,T^k M}$, and that $$ (\phi_q + \mu_T)((u_{-k}^*)_q^{B}) = (u_{-k}^*)_q^{TB} =  (u_{-k+1}^*)_q^{B}.$$ By \ref{hival} and \cite[Lemma 1.6]{GaRo}, we have when $k$ is large that $$ v_\mathfrak{P}((u_{-k}^*)_q^{B}) = (k + \deg(M) - 1-\deg(B))(q-1)-1.$$ In particular, as $\deg(B)$ is determined by $u$ and $w$, and is thus fixed, it follows that the series $$\mathfrak{S} = \sum_{k \in \mathbb{N}} a_{-k} (u_{-k}^*)_q^{B}$$ is also convergent. We let $v_{-k} = (u_{-k}^*)_q^{B}$ for each $k \in \mathbb{N}$. Also by construction, the valuations $v_\mathfrak{P}(v_{-k})$ are strictly increasing as soon as $k + \deg(M) - 1 - \deg(B) > 0$, a bound which depends only on $M$ and $B$, where $\deg(B) \leq \Phi_q(M)$ for the function field analogue $\Phi_q$ of Euler's $\phi$-function. Hence $K$ can be found as desired. This concludes the proof.
\end{proof} 
We note that while inverses of the exponential maps are not uniquely defined, one can construct a well-defined $\mathbb{F}_q$-vector space spanned by the convergent series of Proposition \ref{carlitzconverge}. We now define the completed Carlitz action on $\mathbb{T}_q$:

\begin{defn}\label{def:completecarlitz} For $M = \sum_{k=-\infty}^n a_k T^k \in (\mathbb{F}_q(T))_\infty$, the \emph{completed Carlitz action} is defined on each $u \in \mathbb{T}_q$ as \begin{equation} M \cdot_q u = u_q^M = \sum_{k=0}^n a_k (\phi_q + \mu_T)^k(u) + V_{q,M}(u),\end{equation}
where $V_{q,M}(u)$ is the  $\mathbb{F}_q$-vector space of convergent infinite series $\sum_{k=1}^{\infty} a_{-k} v_{-k}$ appearing in Proposition \ref{carlitzconverge}.
 \end{defn} 
 
 By definition, the completed Carlitz action on $\mathbb{T}_q$ is consistent with the $\mathbb{F}_q[T]$-action. Note that Proposition \ref{carlitzconverge} has only yet been proven for the roots of $q$-unity $\mathbb{T}_q$. We now extend the completed Carlitz exponential action to all of $\mathbb{S}_q$. 
 
 \begin{prop}\label{comcarlitzsq} The completed Carlitz action defines a map from $\mathbb{S}_q$ into itself.
 \end{prop}
 
 \begin{proof} We have already defined the Carlitz action on $\mathbb{T}_q$ in Definition \ref{def:completecarlitz}. Suppose now that $\lambda^* \in \mathbb{S}_q \backslash \mathbb{T}_q$ and $\lambda_m \in \mathbb{T}_q$ $(m \in \mathbb{N})$ are such that $\lambda_m \in \Lambda_{q,A_m}$ for each $m \in \mathbb{N}$ and $\lim_{m \rightarrow \infty} \lambda_m = \lambda^*$. Let $M = \sum_{k=-\infty}^n a_k T^k \in (\mathbb{F}_q(T))_\infty$. By continuity, we have \begin{equation}\label{integerpart} \lim_{m \rightarrow \infty} \sum_{k=0}^n a_k (\phi_q + \mu_T)^k(\lambda_m) = \sum_{k=0}^n a_k (\phi_q + \mu_T)^k(\lambda^*).\end{equation} By definition, any sequence of elements $\xi_m \in V_{q,A_m}(\lambda_m)$ ($m \in \mathbb{N}$) has a convergent subsequence. We may therefore define \begin{equation} M \cdot_q \lambda^* = {\lambda^*}_q^M = \sum_{k=0}^n a_k (\phi_q + \mu_T)^k(\lambda^*) + V_{q}(\lambda^*),\end{equation} where $V_{q}(\lambda^*)$ is the (nonempty) $\mathbb{F}_q$-vector space generated by the collection of limit points of such sequences $\{\xi_m\}_{m \in \mathbb{N}}$, for all such sequences $\{\lambda_m\}_{m \in \mathbb{N}} \subset \mathbb{T}_q$ which converge to $\lambda^*$. \end{proof}
 
 \begin{rem} Note that in both Definition \ref{def:completecarlitz} and Proposition \ref{comcarlitzsq}, if $\lambda \in \mathbb{S}_q$ and $M \in ((\mathbb{F}_q[T])_\infty)$, then $M \cdot_q \lambda = 0$ implies that $V_q(\lambda^*)$ (or $V_{q,M}(\lambda)$, if $\lambda \in \mathbb{T}_q$) is a single element.
 \end{rem}

Via Proposition \ref{comcarlitzsq}, we may finally give the following definition.

\begin{defn} The completed Carlitz action given in Definition~\ref{def:completecarlitz} and Proposition \ref{comcarlitzsq} forms the group of \emph{exponential transformations} of $\mathbb{S}_q$. \end{defn}

 We now move to understanding reciprocity and the natural forms associated with our $q$-unit circle $\mathbb{S}_q$.

\subsection{Reciprocity and algebraic closure} 

 In this section, we begin by relating reciprocity via powers to the cyclotomic function fields, which will be a key determinant of the structure of $\mathbb{S}_q$. In order to build the analogy with the classical reciprocity law more fully, we briefly describe the relationship between splitting over $\mathbb{Q}$ and quadratic reciprocity:
 
In classical quadratic reciprocity, we define the residue symbol for a prime integer $p$ and an integer $a$ with $(a,p)=1$ as $$\left(\frac{a}{p}\right)_L \equiv a^{\frac{p-1}{2}} \pmod{p},$$ which assumes the values of $\pm 1$ depending on whether or not there is an integer solution to the congruence $x^2 \equiv a. \pmod{p}$. The classical splitting lemma is then \cite[Proposition I.8.5]{Neu}:

\begin{lem} \label{classicrecip} Suppose that $p$ is an odd prime integer, and that $a \in \mathbb{Z}$ is square-free with $(a,p)=1$. Then $\left(\frac{a}{p}\right)=1$ if, and only if, $p$ splits completely in the number field $\mathbb{Q}(\sqrt{a})$.	
\end{lem}
Over $\mathbb{F}_q(T)$, there is a similar reciprocity symbol: Suppose that $P \in \mathbb{F}_q[T]$ is monic and irreducible of degree $r$, and that $A \in \mathbb{F}_q[T]$ is such that $(A,P)=1$. Let $d$ be an integer which divides $q-1$. Then we define $$\left(\frac{A}{P}\right)_d \equiv   A^{\frac{q^r-1}{d}} \pmod{P},$$ which, in analogy to the classical symbol, is equal to a unique $\alpha \in \mathbb{F}_q^*$. Such an $\alpha$ is equal to one precisely when there is a solution $F \in \mathbb{F}_q[T]$ to the congruence $F^d \equiv A \pmod{P}$ \cite[Proposition 3.1]{Ros}. We thus obtain the corresponding splitting lemma:

\begin{lem} \label{AmodPd} Let $A \in \mathbb{F}_q(T)$ and $d \in \mathbb{Z}$ which divides $q-1$. Let $P \in \mathbb{F}_q[T]$ be monic and irreducible of degree $r$ such that $(A,P) = 1$. Then $P$ splits completely in $\mathbb{F}_q(T)(^{d}\sqrt{A})$ if, and only if, $$A^{\frac{q^r-1}{d}} \equiv 1 \pmod{P}.$$ In general, $P$ is unramified in $\mathbb{F}_q(T)(^{d}\sqrt{A})$, and the residue degree of $P$ in $\mathbb{F}_q(T)(^{d}\sqrt{A})$ is equal to the multiplicative order of $$A^{\frac{q^r-1}{d}} =\left(\frac{A}{P}\right)_d \pmod{P}.$$
\end{lem}

\begin{proof} By \cite[Proposition 10.5]{Ros}, the Lemma is true whenever $d=l$ is a prime integer over any global function field (not only $\mathbb{F}_q(T)$). The result follows inductively for any prime power $d = l^n$, and then for any such $d$ as in the Lemma by linear disjointness of $\mathbb{F}_q(T)(^{d_1}\sqrt{A})$ and $\mathbb{F}_q(T)(^{d_2}\sqrt{A})$ over $\mathbb{F}_q(T)$ when $(d_1,d_2) = 1$.
\end{proof}

By the classical quadratic reciprocity law $$\left(\frac{p}{q}\right)\left(\frac{q}{p}\right)^{-1} = (-1)^{\frac{p-1}{2}\frac{q-1}{2}}$$ for $p,q \in \mathbb{Z}$ prime with $(p,q)=1$, Lemma \ref{classicrecip} may be used to relate splitting of $p$ in $\mathbb{Q}(\sqrt{q})$ with that of $q$ in $\mathbb{Q}(\sqrt{p})$. One may construct an analogous law for function fields: Suppose that $P,Q \in \mathbb{F}_q[T]$ are monic and irreducible such that $(P,Q)=1$. Then \cite[Theorem 3.3]{Ros} \begin{equation} \label{ffreciprocity} \left(\frac{P}{Q}\right)_d \left(\frac{Q}{P}\right)_d^{-1} = (-1)^{\frac{q-1}{d} \deg(P)\deg(Q)}.\end{equation} 
This is called the \emph{$d$th power reciprocity law} \cite[Chapter 3]{Ros}. By Lemma  \ref{AmodPd}, we may use \eqref{ffreciprocity} to understanding splitting of $P$ in $\mathbb{F}_q(T)(^d\sqrt{Q})$ in terms of that of $Q$ in $\mathbb{F}_q(T)(^d\sqrt{P})$. In particular, when $P$ is linear and $d = q-1$, we obtain: 

\begin{cor} Let $A \in \mathbb{F}_q[T]$. Let $P \in \mathbb{F}_q[T]$ be monic and linear such that $(A,P) = 1$. Then $P$ splits completely in $\mathbb{F}_q(T)(^{q-1}\sqrt{A})$ if, and only if, $$A \equiv 1 \pmod{P}.$$ In general, $P$ is unramified in $\mathbb{F}_q(T)(^{q-1}\sqrt{A})$, and the residue degree of $P$ in $\mathbb{F}_q(T)(^{q-1}\sqrt{A})$ is equal to the multiplicative order of $A \pmod{P}$.
\end{cor}

We now wish to relate this to the roots of $q$-unity, our atoms for $\mathbb{S}_q$. By basic theory, we can find the recipe for splitting in a cyclotomic function field, which conveniently reverses the roles of $A$ and $P$: 

\begin{lem} \label{PmodA} Let $A \in \mathbb{F}_q[T]$. Let $P \in \mathbb{F}_q[T]$ be monic and irreducible such that $(A,P) = 1$. Then $P$ splits completely in $K_{q,A}$ if, and only if, $$P \equiv 1 \pmod{A}.$$ In general, $P$ is unramified in $K_{q,A}$, and the residue degree of $P$ in $K_{q,A}$ is equal to the multiplicative order of $P \pmod{A}$. \end{lem}

\begin{proof} \cite[Corollary 2.5]{Hay1}. \end{proof}

\begin{rem} We note that when the polynomial $Q$ of \eqref{ffreciprocity} is linear, we have $u_q^Q = u^q + Qu$ and $K_{q,Q} = \mathbb{F}_q(T)(^{q-1}\sqrt{-Q})$, which differs from $\mathbb{F}_q(T)(^{q-1}\sqrt{Q})$ only by an element of $\mathbb{F}_{q^2}^*$. \end{rem}

Via the Carlitz action, it is possible to show that the field $V_q$ is algebraically closed relative to the reciprocity law.

\begin{lem} If $M \in (\mathbb{F}_q(T))_\infty$, then the equation $X^{q-1}  - M = 0$ has all of its roots in $V_q$.  \end{lem}

\begin{proof} As $\mu_{q-1} \subset (\mathbb{F}_q(T))_\infty$, it suffices to show that the equation $X^{q-1}  - M = 0$ has at least one root in $V_q$. By definition, $u_q^T = u^q + Tu$, and $\Lambda_{q,T} \subset V_q$. Any $\lambda \in \Lambda_{q,T}$ is a solution to the $$\lambda_q^T = \lambda^q + T\lambda = 0,$$ so that $\lambda^{q-1} = -T$. Let $\mathfrak{P}$ denote the place of $V_q$ above $\mathfrak{p}_\infty$. As $e(\mathfrak{P}|\mathfrak{p}_\infty) = q-1$, it follows that $$(q-1)v_\mathfrak{P}(\lambda) =v_\mathfrak{P}(\lambda^{q-1}) = v_\mathfrak{P}(-T) = v_\mathfrak{P}(T) = (q-1)v_{\mathfrak{p}_\infty} = q-1,$$ hence $v_\mathfrak{P}(\lambda) = 1$. Let $m = v_{\mathfrak{p}_\infty}(M)$, so that $v_{\mathfrak{p}_\infty}((-T)^{-m} M) = 0$, whence $(-T)^{-m} M \in \vartheta_{\mathfrak{p}_\infty}$, where $\vartheta_{\mathfrak{p}_\infty}$ denotes the valuation ring of $\mathbb{F}_q(T)$ at $\mathfrak{p}_\infty$. The reduced equation $$X^{q-1} - \overline{(-T)^{-m} M} \equiv 0 \pmod{\mathfrak{p}_\infty}$$ has a solution $\eta \in \vartheta_{\mathfrak{p}_\infty}/\mathfrak{p}_\infty \cong \mathbb{F}_q^*$. By the generalised Hensel's lemma (Lemma \ref{Henselgeneral}), there exists a unique $u \in  \vartheta_{\mathfrak{p}_\infty}$ such that $u \equiv \eta \pmod{\mathfrak{p}_\infty}$ and $u^{q-1} - (-T)^{-m} M = 0$. It follows that the element $v = \lambda^m u$ satisfies $$v^{q-1} - M = (\lambda^m u)^{q-1} - M = (-T)^m u^{q-1} - M = (-T)^m (u^{q-1} - (-T)^{-m} M) = 0.$$ The result follows. \end{proof}

In classical arithmetic, the product of conjugates of roots of unity lie in $\mathbb{R}$, i.e., if $z \in S^1$, then $z$ is a solution to the quadratic form $Q(x) =x^2 - 2\Re(z)x + |z|\in \mathbb{R}[x]$. We now prove the analogous result for $\mathbb{S}_q$.

\begin{lem} \label{Kummerconverse} Let $u \in \mathbb{S}_q$. Then there exists $M \in (\mathbb{F}_q(T))_\infty$ such that $u^{q-1} - M = 0$. \end{lem}

\begin{proof} Suppose that $A \in \mathbb{F}_q[T]$ is of degree $r$, and let $\lambda \in \Lambda_{q,A}\backslash\{0\}$. The element $\lambda$ is a root of \begin{equation} \label{xi} \frac{u_q^A}{u} = \sum_{i=0}^r \bfrac{A}{i} u^{q^i-1} = \sum_{i=0}^r \bfrac{A}{i} v^{\frac{q^i-1}{q-1}}:=\Xi_{q,A}(v),\end{equation} where $v = u^{q-1}$. As $\deg(\bfrac{A}{i}) = (r-i)q^i$ for each $i=0,\ldots,r$, the Newton polygon for $\Xi_{q,A}(v)$ has vertices contained in \begin{equation}\label{Newtonpoly} \mathcal{N}(\Xi_{q,A})= \left\{\left(\frac{q^i-1}{q-1},-(r-i)q^i\right)\right\}_{i=0}^r,\end{equation} and each of these points lies on the lower convex envelope of this set. It follows that this Newton polygon is equal to the set of line segments connecting adjacent points in $\mathcal{N}(\Xi_{q,A})$. If $\nu$ is a root of $\Xi_{q,A}(v)$, then $v_\mathfrak{P}(\nu)$ is equal to one of the slopes of the Newton polygon of $\Xi_{q,A}(v)$. Hence, there exists $i \in \{1,\ldots,r\}$ such that $$v_\mathfrak{P}(\nu) = (r-i)(q-1) - 1,$$ where $\mathfrak{P}$ is the place of $V_q$ above $\mathfrak{p}_\infty$. As $\nu \in V_q$ and $\dim(V_q/(\mathbb{F}_q(T))_\infty) = q-1$, the minimal polynomial of $\nu$ over $(\mathbb{F}_q(T))_\infty$ has degree $d \mid (q-1)$. But the factor of $\Xi_{q,A}(v)$ corresponding to the segment of the Newton polygon for $\nu$ has degree $\frac{q^{i}-1}{q-1} - \frac{q^{i-1}-1}{q-1} = q^{i-1}$. Thus, $d \mid (q-1,q^{i-1}) = 1$ and $\nu \in (\mathbb{F}_q(T))_\infty$. As $v = u^{q-1}$, there exists a root $\nu$ of $\Xi_{q,A}(v)$ such that $\lambda^{q-1} = \nu$. Thus $\lambda$ is a root of the Kummer equation $X^{q-1} - \nu = 0$, so that the Lemma holds within $\mathbb{T}_q$ by setting $\nu = M$.   

For the completion step, suppose that $\lambda^* \in \mathbb{S}_q$, and let $\{\lambda_n\}_{n \in \mathbb{N}}\subset \mathbb{T}_q$ such that $\lim_{n \rightarrow \infty} \lambda_n = \lambda^*$. By the first part of the proof, for each $n \in \mathbb{N}$, there exists $M_n \in (\mathbb{F}_q(T))_\infty$ such that $\lambda_n^{q-1} = M_n$. As the sequence $\{\lambda_n\}_{n \in \mathbb{N}}$ is Cauchy in $\mathfrak{P} \mid \mathfrak{p}_\infty$, it follows that $\{M_n\}_{n \in \mathbb{N}}$ is Cauchy in $\mathfrak{p}_\infty$. By completeness of $(\mathbb{F}_q(T))_\infty$, the limit $M^* = \lim_{n \rightarrow \infty} M_n$ exists, and $$\left(\lambda^*\right)^{q-1} = \left(\lim_{n \rightarrow \infty} \lambda_n\right)^{q-1}= \lim_{n \rightarrow \infty} \left(\lambda_n^{q-1}\right) = \lim_{n \rightarrow \infty} M_n = M^*,$$ concluding the proof.
\end{proof}

Hence, $V_q$ is the algebraic closure of $(\mathbb{F}_q(T))_\infty$ with respect to Kummer equations (reciprocity), and every element of the $q$-unit circle $\mathbb{S}_q$ corresponds to a residue via the norm map $$N_{V_q/(\mathbb{F}_q(T))_\infty}(u) = \prod_{i=1}^q \zeta^i u,$$ with $\zeta \in \mathbb{F}_q^*$ a primitive $(q-1)$st root of unity. In order words, the Carlitz torsion points and power reciprocity map coincide within $V_q$! 

An essential result in classical cyclotomic theory is:

\begin{lem} Let $\alpha$ be an algebraic integer, all of whose conjugates over $\mathbb{Q}$ lie on $S^1$. Then $\alpha \in \mathbb{M}$.
\end{lem}

\begin{proof}\cite[Lemma 1.6]{Wash}.
\end{proof}

In particular, this allows one to study the structure of conjugates of elements of $S^1$ which are not roots of unity, as we know that such elements have at least one conjugate which does not lie on $S^1$. We now show that $\mathbb{S}_q$ possesses a similar property relative to its set of roots of $q$-unity $\mathbb{T}_q$.

\begin{lem}\label{sqtq}
Suppose that $\lambda \in \mathbb{S}_q$ is integral over $\mathbb{F}_q[T]$ and all conjugates of $\lambda$ over $\mathbb{F}_q(T)$ are contained in $\mathbb{S}_q$. Then $\lambda \in \mathbb{T}_q$. 
\end{lem}

\begin{proof} If $\sigma$ is in the Galois group of $\lambda$ over $\mathbb{F}_q(T)$ and $M \in \mathbb{F}_q[T]$, then $\lambda_q^M \in \mathbb{F}_q[T,\lambda]$ and \begin{equation} \label{Galoiscommute} \sigma\left( \lambda_q^M\right) = \sigma\left( \sum_{i=0}^{\deg(M)} \bfrac{M}{i} \lambda^{q^i}\right) = \sum_{i=0}^{\deg(M)} \bfrac{M}{i} \sigma(\lambda)^{q^i},\end{equation} so that $\sigma$ commutes with the Carlitz action. As $\mathbb{T}_q$ lies in the separable closure of $\mathbb{F}_q(T)$, $\lambda$ is separable over $\mathbb{F}_q(T)$. Let $$F(X) = \prod_{i=1}^r (X - \lambda_i) \in \mathbb{F}_q[T][X]\qquad (\lambda = \lambda_1)$$ be the minimal polynomial of $\lambda$ over $\mathbb{F}_q(T)$. By \eqref{Galoiscommute}, the elements $(\lambda_1)_q^M, \ldots,$ $(\lambda_r)_q^M$ are the conjugates of $\lambda_q^M$. As $\lambda_1,\ldots,\lambda_r \in \mathbb{S}_q$, it follows by Proposition \ref{sqvectorspace} that $(\lambda_1)_q^M, \ldots, (\lambda_r)_q^M \in \mathbb{S}_q$ for all $M \in \mathbb{F}_q[T]$. The polynomial $F_M(X) = \prod_{i=1}^r (X - (\lambda_i)_q^M)$ has $\lambda_q^M$ as a root, and as $\lambda_q^M$ is integral over $\mathbb{F}_q[T]$, $F_M(X) \in \mathbb{F}_q[T][X]$. Let $\mathfrak{P}$ be the place of $\mathbb{S}_q$ above $\mathfrak{p}_\infty$. Let $\{u_{1,n} \}_{n \in \mathbb{N}} \subset \mathbb{T}_q$ such that $\lim_{n \rightarrow \infty} u_{1,n}  = \lambda \; (=\lambda_1)$. By \cite[Lemma 1.5]{GaRo}, there exists a nonnegative integer $k_{1,n} $ such that $v_\mathfrak{P}(u_{1,n} ) = k_{1,n} (q-1) - 1$. By the non-Archimedean property, it follows that there exists a nonnegative integer $k_1$ such that $v_\mathfrak{P}(\lambda_1) = k_1(q-1) - 1$. Let $M \in \mathbb{F}_q[T]$. By definition of the Carlitz action, we then obtain $\lim_{n \rightarrow \infty} (u_{1,n})_q^M = \lambda_1^M$. As $(u_{1,n})_q^M \in \mathbb{T}_q$ for all $n \in \mathbb{N}$, there exists a nonnegative integer $k_{1,n,M}$ such that $v_\mathfrak{P}((u_{1,n})_q^M) = k_{1,n,M} (q-1) - 1$, and again by the non-Archimedean property, there exists a nonnegative integer $k_{1,M}$ such that $v_\mathfrak{P}((\lambda_{1})_q^M) = k_{1,M} (q-1) - 1$. The same argument holds for $\lambda_2,\ldots,\lambda_r$ in place of $\lambda_1 = \lambda$, so that \begin{equation} \label{valuation} v_\mathfrak{P}((\lambda_{i})_q^M) = k_{i,M} (q-1) - 1\qquad (k_{i,M} \geq 0,\; i=1,\ldots,r).\end{equation}
Let $$F_M(X) = X^r + a_{r-1,M} X^{r-1} + \cdots + a_{1,M} X + a_{0,M},$$ and let $E_j(X_1,\ldots,X_r)$ denote the $j$th symmetric elementary function. By definition of $F_M(X)$, we have $a_{r-j,M} = E_j((\lambda_{1})_q^M,\ldots,(\lambda_{r})_q^M))$. As $F_M(X) \in \mathbb{F}_q[T][X]$, we also have $v_\mathfrak{P}(a_{r-j,M}) \leq 0$ for each $j=0,\ldots,r$. We thus obtain for each $j=0,\ldots,r$ that \begin{align} 0 &\notag \geq v_\mathfrak{P}(a_{r - j,M}) \\\label{valbound}&= v_\mathfrak{P}(E_{j}((\lambda_{1})_q^M,\ldots,(\lambda_{r})_q^M)))\\\notag& \geq \min_{i_1,\ldots,i_{j}} v_\mathfrak{P}\left(\prod_{t=1}^{j} (\lambda_{i_t})_q^M\right) \\\notag& \geq -j.\end{align} As $a_{r-j,M} \in \mathbb{F}_q[T]$, we obtain $v_\mathfrak{P}(a_{r - j,M}) = (q-1)v_{\mathfrak{p}_\infty}(a_{r - j,M})$, so that by \eqref{valbound}, $a_{r-j,M}$ is a polynomial in $\mathbb{F}_q[T]$ of degree at most $\frac{j}{q-1}$. It follows that the finite set $$\mathcal{F} = \left\{F \in \mathbb{F}_q[T][X]\;\bigg|\; F = \sum_{j=0}^r b_{r-j} X^j, \; \deg(b_{r-j}) \leq \frac{j}{q-1}\right\},\; \left|\mathcal{F}\right|=\prod_{j=0}^r q^{\frac{j}{q-1}+1}$$ contains $F_M(X)$ for each $M \in \mathbb{F}_q[T]$. As the set of polynomials in $\mathcal{F}$ possesses finitely many roots and $\lambda_q^M$ is among them for each $M \in \mathbb{F}_q[T]$, it follows that there exist $M,N \in \mathbb{F}_q[T]$ such that $M \neq N$ and $\lambda_q^M = \lambda_q^N$. By definition of the Carlitz action, we thus obtain $$\lambda_q^{M-N} = \lambda_q^M - \lambda_q^N =0,$$ so that $\lambda \in \Lambda_{q,M-N} \subset \mathbb{T}_q$, as claimed. 
 \end{proof}
 
By definition of the Carlitz module, all elements of $ \mathbb{T}_q$ are integral over $\mathbb{F}_q[T]$ and take all of their conjugates over $\mathbb{F}_q(T)$ in $ \mathbb{T}_q$, so the converse of Lemma \ref{sqtq} also holds. Also, the elements of $\kappa(\mathbb{T}_q)$ are integral over $\mathbb{F}_q[T]$: The field $\mathbb{F}_q(T)(\kappa(\mathbb{T}_q))$ is the compositum of all of the totally real subfields of the cyclotomic function fields $K_{q,M}/\mathbb{F}_q(T)$ \cite[Definition 12.5.5]{Vil}. 

 \begin{cor} If $\lambda \in \kappa(\mathbb{S}_q)$ is integral over $\mathbb{F}_q[T]$ and all conjugates of $\lambda$ over $\mathbb{F}_q(T)$ are contained in $\kappa(\mathbb{S}_q)$, then $\lambda \in \kappa(\mathbb{T}_q)$. 
 \end{cor}

\begin{proof}
This follows immediately from Lemma \ref{sqtq} and the definition of $\kappa$.	
\end{proof}

\subsection{A pinch of topology} From \S 3.4, we observed that the $q$-unit circle $\mathbb{S}_q$ consists of solutions to the map. $\kappa : \mathbb{S}_q \rightarrow (\mathbb{F}_q(T))_\infty$ defined as $\kappa(u) = u^{q-1}$. In classical theory, the product of conjugates of any root of unity $z \in S^1$ is equal to one. We examine what the analogy to this is in $\mathbb{S}_q$ by description of the image of the map $\kappa$. For this, we first give some preliminary results on the topology of $\mathbb{S}_q$ and $\kappa(\mathbb{S}_q)$, from which we conclude that the $q$-unit circle is topologically similar to $S^1$. 

The following result can be viewed as the analogue to the statement that if a sequence $\{e^{i p_n/q_n}\}_{n \in \mathbb{N}}$ converges to $e^{i\theta}$ where $\theta \notin \mathbb{Q}$, then $\lim_{n \rightarrow \infty} p_n, q_n = \infty$. 

\begin{lem}\label{degreetoinfty}
Suppose that $\lambda^* \in \mathbb{S}_q \backslash \mathbb{T}_q$ is the limit of $\{\lambda_n\}_{n \in \mathbb{N}} \subset \mathbb{T}_q$. Letting $A_n \in \mathbb{F}_q[T]$ be such that $\lambda_n \in \Lambda_{q,A_n}$ for each $n \in \mathbb{N}$, then $\lim_{n \rightarrow \infty} \deg(A_n) = \infty$.
\end{lem}

\begin{proof} Suppose that $C > 0$ is such that $\deg(A_n) \geq C$ for all $n \in \mathbb{N}$. As the set of polynomials in $\mathbb{F}_q[T]$ with degree at most $C$ is finite, there must be a subsequence $\{A_{n_k}\}_{k \in \mathbb{N}}$ of $\{A_n\}_{n \in \mathbb{N}}$ such that $A_{n_1} = A_{n_2} = \cdots := A^*$. By definition of the Carlitz action, $$(\lambda_{n_k})_{q}^{A^*} = \sum_{i=0}^{\deg(A^*)} \bfrac{A^*}{i} \lambda_{n_k}^{q^i} = 0.$$ As $\lim_{k \rightarrow \infty} \lambda_{n_k} = \lambda^*$, we obtain \begin{align*} (\lambda^*)_{q}^{A^*} &= \sum_{i=0}^{\deg(A^*)} \bfrac{A^*}{i} {\lambda^*}^{q^i} \\&= \sum_{i=0}^{\deg(A^*)} \bfrac{A^*}{i} {\left(\lim_{k \rightarrow \infty} \lambda_{n_k}\right)}^{q^i} \\&= \lim_{k \rightarrow \infty} \sum_{i=0}^{\deg(A^*)} \bfrac{A^*}{i} \lambda_{n_k}^{q^i} \\& = 0.\end{align*} Hence $\lambda^* \in \Lambda_{q,A^*} \subset \mathbb{T}_q$, a contradiction.\end{proof}

One can ask subtle questions about speed of convergence to limit points in $\mathbb{S}_q$ as a result of Lemma \ref{degreetoinfty}. In the following result, we prove the analogue of the Dirichlet approximation theorem.

\begin{lem}[Dirichlet approximation theorem] Suppose that $\lambda^* \in \mathbb{S}_q \backslash \mathbb{T}_q$. Let $\lambda^* = i(q-1) - 1$. Then for each $n \in \mathbb{N}$ with $n \geq i+1$, there exist $M_n \in \mathbb{F}_q[T]$ such that $\deg(M_n) = n$ and $\lambda_n \in \Lambda_{q,M_n}$ such that $$v_\mathfrak{P}(\lambda_n - \lambda^*) > (n-1)(q-1) -1.$$
\end{lem}

\begin{proof} As $\lambda^* \in \mathbb{S}_q$, there exists a nonnegative integer $i$ such that \begin{equation}\label{starval} v_\mathfrak{P}(\lambda^*) = i(q-1) - 1.\end{equation} By \cite[Lemma 1.5]{GaRo}, any $\lambda \in \mathbb{T}_q$ for which $v_\mathfrak{P}(\lambda- \lambda^*) > i(q-1) - 1$ must lie in some $\Lambda_{q,M}$ where $\deg(M) > i$. We wish to minimise the degree of such an $M$. Suppose that $\deg(M) = i+1$. By examination of the Newton polygon \eqref{Newtonpoly} in the proof of Lemma \ref{Kummerconverse}, there exist $q-1$ elements of $\Lambda_{q,M}$ with valuation equal to $i(q-1) - 1$ at $\mathfrak{P}$, and each of these are multiplicative conjugates over $\mathbb{F}_q^*$. As the residue field at $\mathfrak{P}$ is isomorphic to $\mathbb{F}_q$ and $\lambda^*$ also has valuation $i(q-1)-1$ \eqref{starval}, it follows that one of these such $\lambda\in \Lambda_{q,M}$ satisfies $$\lambda \equiv\lambda^* \pmod{\mathfrak{P}^{i(q-1)}}.$$ If $\deg(M) = i+2$, then again by the Newton polygon \eqref{Newtonpoly}, there are $q(q-1)$ elements of $\Lambda_{q,M}$ with valuation equal to $i(q-1) - 1$ at $\mathfrak{P}$. (These correspond to exactly $q$ distinct Kummer polynomial factors, each of degree $q-1$ over $((\mathbb{F}_q[T])_\infty)$.) As $\Lambda_{q,M}$ is an additive group and $\deg(M)=i+2$, it follows by \cite[Lemma 1.5]{GaRo} that, if two of these $q(q-1)$ elements $\lambda_1,\lambda_2\in\Lambda_{q,M}$ satisfy $$v_\mathfrak{P}(\lambda_2 - \lambda_1)> (i+1)(q-1) - 1,$$ then $\lambda_1 = \lambda_2$. Furthermore, if $v_\mathfrak{P}(\lambda_2 - \lambda_1)> i(q-1) - 1$, then as $\Lambda_{q,M}$ is an additive group, $\lambda_2 - \lambda_1 \in \Lambda_{q,M}$, so that again by \cite[Lemma 1.5]{GaRo}, we obtain $$v_\mathfrak{P}(\lambda_2 - \lambda_1) = j(q-1) - 1$$ for some nonnegative integer $j$, from which it follows that $v_\mathfrak{P}(\lambda_2 - \lambda_1)\geq (i+1)(q-1) - 1$. Also, for any $\lambda \in\Lambda_{q,M}$, we have \begin{equation} \label{stardiffval} v_\mathfrak{P}(\lambda -\lambda^*)=k(q-1)-1\end{equation} for some nonnegative integer $k$, as one easily sees by letting $\{\mu_n\}_{n\in\mathbb{N}} \subset \mathbb{T}_q$ be such that $\lim_{n\rightarrow\infty} \mu_n=\lambda^*$, so that $\lambda-\mu_n \in \Lambda_{q,M}$ and $v_\mathfrak{P}(\lambda-\mu_n) = k_n(q-1)-1$, which by the non-Archimedean property is eventually equal to $v_\mathfrak{P}(\lambda -\lambda^*)$.

We have already shown that there are $q(q-1)$ elements of $\Lambda_{q,M}$ which occupy distinct classes in $\mathfrak{P}^{i(q-1)-1}/\mathfrak{P}^{(i+1)(q-1)}$. As such $\lambda_1,\lambda_2$ satisfy $$v_\mathfrak{P}(\lambda_2 - \lambda_1)> i(q-1) - 1\Rightarrow v_\mathfrak{P}(\lambda_2 - \lambda_1)\geq (i+1)(q-1) - 1,$$ then by \eqref{stardiffval} and comparing cardinalities, there must be some $\lambda$ among these $q(q-1)$ elements for which $v_\mathfrak{P}(\lambda - \lambda^*) > (i+1)(q-1) - 1$. This concludes the proof for $\deg(M)=i+2$. 

Via the Newton polygon \eqref{Newtonpoly}, one may again proceed similarly for higher values of $\deg(M)$. 
\end{proof}

Potential refinements of this are left as an exercise to the interested reader. By continuity of $\kappa$ and the proofs of Lemmas \ref{Kummerconverse} and \ref{degreetoinfty}, we may give an explicit description of $\kappa(\mathbb{S}_q)$ in terms of polynomials. In the following result, notice the interesting deviation from the classical $S^1$, for which the norm map sends every element to one. This is a natural consequence of the fact that the norm remains multiplicative, whereas the Carlitz action is $\mathbb{F}_q[T]$-additive.

\begin{cor} The set $\kappa(\mathbb{S}_q)$ is the closure at $\mathfrak{p}_\infty$ of the collection of roots of the polynomials $$\Xi_{q,A}(v) =\sum_{i=0}^{\deg(A)} \bfrac{A}{i} v^{\frac{q^i-1}{q-1}},\qquad (A \in \mathbb{F}_q[T])$$ which were introduced in \eqref{xi}.  If $v^* \in \kappa(\mathbb{S}_q\backslash \mathbb{T}_q)$ and $\{v_n\}_{n \in \mathbb{N}} \subset \kappa(\mathbb{T}_q)$ satisfies $\lim_{n \rightarrow \infty} v_n = v^*$, then each $v_n$ is a root of  $\Xi_{q,A_n}(v)$, where $A_n \in \mathbb{F}_q[T]$ is such that $\lim_{n \rightarrow \infty} \deg(A_n) = \infty$. 
\end{cor}

We now give the topological description of $\mathbb{S}_q$, beginning with density:

\begin{lem} \label{sqnowheredense} The unit circle $\mathbb{S}_q$ is nowhere dense in $V_q$. \end{lem}

\begin{proof} By definition, $\mathbb{S}_q$ is closed, so it is enough to show that the complenent $\mathbb{S}_q^c = V_q \backslash \mathbb{S}_q$ is dense in $V_q$. To see this, let $u \in \mathbb{S}_q$, and let $\{M_n\}_{n \in \mathbb{N}} \subset (\mathbb{F}_q(T))_\infty$ be chosen such that $M_n \neq 0$ for all $n \in \mathbb{N}$ and $\lim_{n \rightarrow \infty} M_n = 0$. By Lemma \ref{infinitysq}, $M_n \notin \mathbb{S}_q$ for any $n \in \mathbb{N}$, from which it follows that  $\{u + M_n\}_{n \in \mathbb{N}} \cap \mathbb{S}_q = \varnothing$. Clearly $\lim_{n \rightarrow \infty} (u + M_n) = u$, concluding the proof. \end{proof}

The space $V_q$ is naturally endowed with a metric according to the valuation at $\mathfrak{P}$. As is true for $S^1$, we may show that the $q$-unit circle $\mathbb{S}_q$ is compact.

\begin{lem}[Compactness] \label{sqcompact} $\mathbb{S}_q$ is compact in the $\mathfrak{P}$-metric. 
\end{lem}

\begin{proof} Let $\{\lambda_n\}_{n \in \mathbb{N}}$ be a sequence in $\mathbb{S}_q$. We wish to show that this has a convergent subsequence. By definition, each $\lambda_n$ is a limit of a sequence $\{\lambda_{n,k}\}_{k \in \mathbb{N}}$, and by \cite[Lemma 1.5]{GaRo}, it follows that there exists an integer $i_{n,k} \geq 0$ such that $v_\mathfrak{P}(\lambda_{n,k}) = i_{n,k}(q-1) - 1$. Hence by the non-Archimedean property, there exists for each $n \in \mathbb{N}$ some $i_n \geq 0$ such that $v_\mathfrak{P}(\lambda_{n}) = i_{n}(q-1) - 1$. Clearly if $\lim_{n \rightarrow \infty} i_n = \infty$, then $\lim_{n \rightarrow \infty} \lambda_{n} = 0$, and the sequence $\{\lambda_n\}_{n \in \mathbb{N}}$ is thus convergent. We may therefore assume that the set $\{i_n\}_{n \in \mathbb{N}}$ is bounded. As $i_n \geq 0$ for all $n \in \mathbb{N}$, we find a subsequence $\{\lambda_{n_j}\}_{j \in \mathbb{N}}$ such that $i_{n_1} = i_{n_2} = \cdots = i^*$. As the residue field is isomorphic to $\mathbb{F}_q$, we can pass to a subsequence of $\{\lambda_{n_j}\}_{j \in \mathbb{N}}$, all of which belong to the same residue class modulo $\mathfrak{P}^{i^*(q-1) -1}$. Continuing inductively, we may find a subsequence $\{\lambda_n^*\}_{n \in \mathbb{N}}$ for which $$v_\mathfrak{P}(\lambda_2^* - \lambda_1^*) < v_\mathfrak{P}(\lambda_3^* - \lambda_2^*) < \cdots $$ In particular, it follows that $\{\lambda_n^*\}_{n \in \mathbb{N}}$ is Cauchy, hence convergent by completeness of $\mathbb{S}_q$.\end{proof}

We now wish to show that $\mathbb{S}_q$ is of measure zero within $V_q$. Of course, we have not yet, nor will we, define a measure. For our results, we will only need to assume that whatever measure we define is finite on compact sets in the $\mathfrak{P}$-metric.

\begin{lem} \label{sqmeasurezero} Let $\mu$ be a measure on $V_q$ such that $\mu(K) < \infty$ for every compact set $K \subset V_q$. Then $\mu(\mathbb{S}_q) = 0$. 
\end{lem}

\begin{proof} For a contradiction, suppose $\mu(\mathbb{S}_q) > 0$. For each $\lambda \in \mathbb{S}_q$, we have for some (nonnegative) integer $i$ that $v_\mathfrak{P}(\lambda) = i(q-1)-1 \equiv -1 \pmod{q-1}.$ Let $\lambda_0 \in \Lambda_{q,T}\backslash \{0\}$, so that $\lambda_0^{q-1} = -T$, and hence $\lambda_0^{-1}$ is a prime element of $V_q$. Consider the collection of series of the form \begin{equation} \label{vseries}v = \sum_{k=0}^\infty a_k \lambda_0^{-k}, \qquad a_k \in \mathbb{F}_q,\;\; k=0,1,\ldots.\end{equation} Such a sequence is convergent and is thus and element of $\mathbb{S}_q$. Provided that the smallest value of $k$ such that $a_k \neq 0$ satisfies $k \neq -1 \pmod{q-1}$, we have $v \notin \mathbb{S}_q$. Provided that $q > 2$, there is an uncountable infinity of such series $v$. By \eqref{disjoint}, for each such $v$, we have $\mathbb{S}_q(v) \cap \mathbb{S}_q = \varnothing$. By definition, $v_\mathfrak{P}(v) \geq 0$, and for each $\lambda \in \mathbb{S}_q$, we have $v_\mathfrak{P}(\lambda) \geq -1$. Clearly the set of series of the form \eqref{vseries}, equal to the valuation ring $\varnothing_\mathfrak{P}$ at $\mathfrak{P}$, is compact, and by Lemma \ref{sqcompact}, it follows that the union $$\mathbb{S}_q(\vartheta_\mathfrak{P}) = \bigcup_{v \in \vartheta_\mathfrak{P}}\mathbb{S}_q(v).$$ is also compact. Hence $\mu(\mathbb{S}_q(\vartheta_\mathfrak{P})) < \infty$ But $\mathbb{S}_q(\vartheta_\mathfrak{P})$ contains infinitely many copies of disjoint translations of $\mathbb{S}_q$, contradicting $\mu(\mathbb{S}_q) > 0$.
	\end{proof}

The set $\kappa(\mathbb{S}_q)$ admits a natural $\mathbb{F}_q[T]$-action: Let $\nu \in \kappa(\mathbb{T}_q)$ and $u \in \mathbb{T}_q$ such that $u^{q-1} = \nu$. Let $A \in \mathbb{F}_q[T]$. We define $$A \star_q \nu := \kappa(u_q^A) = (u_q^A)^{q-1}.$$ As $u_q^A$ is a polynomial depending only on $A$, we may pass to completions to define an action of $\mathbb{F}_q[T]$ on all of $\kappa(\mathbb{S}_q)$. We note that as $\kappa(\mathbb{S}_q)$ is the image of $\mathbb{S}_q$ via a non-additive polynomial, it admits the $\mathbb{F}_q[T]$-action $\star_q$ but does not possess an $\mathbb{F}_q[T]$-module structure. We now describe $\kappa(\mathbb{S}_q)$ topologically with the help of Lemma \ref{sqnowheredense}: 

\begin{cor} The set $\kappa(\mathbb{S}_q)$ is nowhere dense and compact in $(\mathbb{F}_q(T))_\infty$. If $\mu$ is a measure on $(\mathbb{F}_q(T))_\infty$ such that $\mu(K) < \infty$ for every compact set $K \subset (\mathbb{F}_q(T))_\infty$, then $\mu(\kappa(\mathbb{S}_q)) = 0$.  \end{cor} 

\begin{proof} By construction, the set $\kappa(\mathbb{S}_q)$ is closed, so it suffices to show that $\kappa(\mathbb{S}_q)^c = (\mathbb{F}_q(T))_\infty \backslash \kappa(\mathbb{S}_q)$ is dense in $(\mathbb{F}_q(T))_\infty$. To see this, let $\nu \in \kappa(\mathbb{S}_q)$ and $u \in \mathbb{S}_q$ such that $\kappa(u) = \nu$. By Lemma \ref{sqnowheredense}, we may find a sequence $\{u_n\}_{n \rightarrow \infty}$ not belonging to $\mathbb{S}_q$ and converging to $u$. Let $\nu_n:= \kappa(u_n)$. By continuity of $\kappa$, we obtain $\lim_{n \rightarrow \infty} \nu_n = \lim_{n \rightarrow \infty} \kappa(u_n) =\kappa(u) = \nu$. 

We must now show that the sequence $\{\nu_n\}_{n \rightarrow \infty}$ does not belong to $\kappa(\mathbb{S}_q)$. Suppose that $\nu_n \in \kappa(\mathbb{S}_q)$ for some $n \in \mathbb{N}$. Then there exists $w_n \in \mathbb{S}_q$ such that $\kappa(w_n) = \nu_n$. But then $$u_n^{q-1} = \kappa(u_n) = \nu_n = \kappa(w_n) = w_n^{q-1},$$ so that $(u_n/w_n)^{q-1} = 1$ and there exists $\alpha \in \mathbb{F}_q^*$ such that $u_n = \alpha w_n$. By Proposition \ref{sqvectorspace}, $\mathbb{S}_q$ is an $\mathbb{F}_q$-vector space, which implies that $u_n \in \mathbb{S}_q$, a contradiction.

Compactness follows from Lemma \ref{sqcompact} and continuity of $\kappa$. For the proof that the $\mu$-measure of $\kappa(\mathbb{S}_q) = 0$, note that if $v \in \kappa(\mathbb{S}_q)$, then $v = \kappa(u) = u^{q-1}$ for some $u \in \mathbb{S}_q$, whence as the ramification index $e(\mathfrak{P}|\mathfrak{p}_\infty) = q-1$, we find $$(q-1) v_\mathfrak{P}(u) = v_\mathfrak{P}(u^{q-1}) = v_\mathfrak{P}(v) = (q-1) v_{\mathfrak{p}_\infty}(v),$$ so that $v_{\mathfrak{p}_\infty}(v) = v_\mathfrak{P}(u)$. Thus, the same argument as in Lemma \eqref{sqmeasurezero} may be used with $T^{-1}$ in place of $\lambda^{-1}$. \end{proof}

\subsection{Farey fractions and the Bruhat-Tits building}

In the classical Farey-Ford circle packing, horoballs are based at the points of $\mathbb{P}_1(\mathbb{Q})$. Two such horoballs associated with $\frac{a}{b},\frac{c}{d} \in \mathbb{Q}$ are tangent if, and only if, $|ad - bc| = 1$ \cite[p. 27]{Con}. In other words, $ad - bc$ must be a unit in $\mathbb{Z}$. Conway's topograph is a 3-regular tree, which naturally embeds in the Poincar\'{e} disk (and the hyperbolic plane) via the Farey-Ford packing (Ibid.) and realises the abstract tree \cite[Chapter I.4.2]{Ser} associated with $$\mathrm{SL}_2(\mathbb{Z}) \cong (\mathbb{Z}/4\mathbb{Z}) \ast_{(\mathbb{Z}/2\mathbb{Z})} (\mathbb{Z}/6\mathbb{Z}).$$ If $C_1,C_2,C_3$ are mutually tangent circles in the packing associated with fractions $\frac{a}{b},\frac{c}{d},\frac{e}{f}$, then the radii of each is given by $r_1 = \frac{1}{2b^2}, r_2=\frac{1}{2d^2},r_3 = \frac{1}{2f^2}$, whence their curvatures ($k_i=1/r_i$) are equal to $2b^2, 2d^2,2f^2$. This triple satisfies Descartes' relation \cite{Des} $$(k_1 + k_2 + k_3)^2 = 2(k_1^2 + k_2^2 + k_3^2).$$ The same is true for any three mutually tangent circles on a line.

In $\mathbb{F}_q(T)$, two horoballs associated with $\frac{P}{Q},\frac{R}{S} \in \mathbb{F}_q(T)$ are tangent if, and only if, $PS - QR$ is a unit in $\mathbb{F}_q[T]$, i.e., an element of  $\mathbb{F}_q^*$. The topograph is the Bruhat-Tits tree $\mathcal{T}_{q+1}$ of $G_{2,q} = \mathrm{SL}_2((\mathbb{F}_q(T))_\infty)$, which is \emph{$(q+1)$-valent} \cite{Kim}. As we have seen in \S 3.2, the $q$-unit circle $\mathbb{S}_q$ generates a vector space of dimension $q-1$ over the real line $(\mathbb{F}_q(T))_\infty$, and from \S 3.4, the Kummer map $\kappa(u) = u^{q-1}$ gives the reciprocity law for $\mathbb{F}_q(T)$. In terms of horoballs, we obtain:

\begin{lem}[Descartes relation] \label{tangentcircles} Let $C_1,\ldots,C_{q+1}$ be mutually tangent horoballs associated with curvatures $\frac{P_1}{Q_1}, \frac{P_2}{Q_2},\ldots,\frac{P_{q+1}}{Q_{q+1}} \in \mathbb{P}^1(\mathbb{F}_q(T))$. Then the curvatures of $C_1,\ldots,C_{q+1}$ are solutions to the form $$\mathcal{K}(X_1,\ldots,X_{q+1}) = (X_1 + \cdots + X_{q+1})^{q-1} - (X_1^{q-1} + \cdots + X_{q+1}^{q-1}).$$
\end{lem}

\begin{proof} Let $\frac{P_1}{Q_1}, \frac{P_2}{Q_2},\ldots,\frac{P_{q+1}}{Q_{q+1}}$ be as in the statement of the theorem. We may write \begin{equation}\label{transformation}\frac{P_j}{Q_j} = \frac{P_j Q_1}{Q_j Q_1} = \frac{P_1 Q_j + \alpha_j}{Q_j Q_1}. \qquad \alpha_j \in \mathbb{F}_q^*, \; j=2,\ldots,q+1.\end{equation} Without loss of generality, we may remove $Q_1$ from denominators, as it appears in every term of $\mathcal{K}$ after the transformation in \eqref{transformation}. Set $\alpha_1=0$. We may write \begin{equation}\label{summand}	
\sum_{j=1}^{q+1} \frac{P_1 Q_j + \alpha_j}{Q_j } = \sum_{j=1}^{q+1} \left({P_1} + \frac{\alpha_j}{Q_j}\right).\end{equation} Hence we wish to show that $\sum_{j=1}^{q+1} (P_1 + \frac{\alpha_j}{Q_j})$ is a solution to the equation $$X^{q-1} = \sum_{j=1}^{q+1} \left(P_1 + \frac{\alpha_j}{Q_j}\right)^{q-1}:= M,$$ i.e., $\lambda$ is a $(q-1)$st root of $M$. (Note that $M \in \mathbb{P}^1(\mathbb{F}_q(T))$.) We may write \begin{equation}\label{powersinside} \sum_{j=1}^{q+1} \left(P_1 + \frac{\alpha_j}{Q_j}\right)^{q-1} = \sum_{j=1}^{q+1} \sum_{k=0}^{q-1} {q-1 \choose k} P_1^k \left(\frac{\alpha_j}{Q_j}\right)^{q-1-k}.\end{equation} On the other hand, we have \begin{align}\label{powersoutside} \notag &\left[\sum_{j=1}^{q+1} \left(P_1 + \frac{\alpha_j}{Q_j}\right)\right]^{q-1} = \sum_{k_1,\ldots,k_{q+1}} {q-1 \choose k_1,\ldots,k_{q+1}} \prod_{j=1}^{q+1} \left(P_1 + \frac{\alpha_j}{Q_j}\right)^{k_j}  \\\notag& = \sum_{k_1,\ldots,k_{q+1}} {q-1 \choose k_1,\ldots,k_{q+1}} \prod_{j=1}^{q+1} \left[\sum_{l=0}^{k_j} {k_j \choose l} P_1^l  \left(\frac{\alpha_j}{Q_j}\right)^{k_j-l} \right]\\& =  \sum_{k_1,\ldots,k_{q+1}} {q-1 \choose k_1,\ldots,k_{q+1}} \sum_{l_1 \leq k_1,\ldots,l_{q+1}\leq k_{q+1}}\prod_{j=1}^{q+1}  {k_j \choose l_j} P_1^{\sum_{j} l_j}  \left(\frac{\alpha_j}{Q_j}\right)^{k_j-l_j} \\\notag & =  \sum_{k_1,\ldots,k_{q+1}} \sum_{l_1 \leq k_1,\ldots,l_{q+1}\leq k_{q+1}} P_1^{\sum_{j} l_j} {q-1 \choose k_1,\ldots,k_{q+1}} \prod_{j=1}^{q+1}  {k_j \choose l_j}  \left(\frac{\alpha_j}{Q_j}\right)^{k_j-l_j}  \\\notag & =  \sum_{l_1,\ldots,l_{q+1}}  P_1^{\sum_{j} l_j} \sum_{ k_1 \geq l_1,\ldots,k_{q+1}\geq l_{q+1}}   {q-1 \choose l_1,\ldots,l_{q+1},k_1-l_1,\ldots,k_j - l_j} \prod_{j=1}^{q+1} \left(\frac{\alpha_j}{Q_j}\right)^{k_j-l_j}.\end{align} Clearly the terms in \eqref{powersinside} also appear in \eqref{powersoutside}, corresponding to those terms in \eqref{powersoutside} where $k_j = 0$ for all except a single index $j$. As $C_1,\ldots,C_{q+1}$ are mutually tangent, the values of $\alpha_j$ run through $\mathbb{F}_q$, with a single duplicate value coming from the degree function. Let us represent this duplicate by $\alpha_{q+1}$. Then there exists $F,G \in \mathbb{F}_q[T]$ such that \begin{equation} \label{FandG} P_j = P_1 + \beta_j F,\;\; Q_j = Q_1 + \beta_j G, \quad \beta_j \in \mathbb{F}_q,\;\; j=1,\ldots,q.\end{equation} By \eqref{FandG}, we obtain for each $j=1,\ldots,q$ that
$$\alpha_j = P_j Q_1 - P_1 Q_j = (P_1 + \beta_j F) Q_1 - P_1 (Q_1 + \beta_j G) = \beta_j(FQ_1 - P_1 G).$$ As the quantities $\alpha_j$ ($j=1,\ldots,q$) run through all of $\mathbb{F}_q$, so do the coefficients $\beta_j$. Via a tedious calculation involving evaluation of symmetric functions of $q$ variables assuming each of the values $\{\beta_1,\ldots,\beta_q\} = \mathbb{F}_q$, one can use \eqref{powersoutside} and \eqref{FandG} to show that, after removal of the terms from \eqref{powersinside}, the inner sum in the last line of \eqref{powersoutside} vanishes modulo $p$ for each fixed value of $\sum_j l_j$, the exponent of $P_1$ in the outer sum. \end{proof}

The following Corollary is immediate from Lemma \ref{tangentcircles}:

\begin{cor} If $C_1,\ldots,C_{q+1}$ are mutually tangent horoballs with curvatures 	$\frac{P_1}{Q_1}, \frac{P_2}{Q_2},\ldots,\frac{P_{q+1}}{Q_{q+1}} \in \mathbb{P}^1(\mathbb{F}_q(T))$, then the polynomial $$f(X) = X^{q-1} - M,\qquad M = \sum_{j=1}^{q+1} \left(P_1 + \frac{\alpha_j}{Q_j}\right)^{q-1}$$ splits over $\mathbb{P}^1((\mathbb{F}_q(T))_\infty)$.
\end{cor}

\begin{rem} In higher dimensions $n > 2$ over the real line, Soddy's mutually tangent spheres \cite{Sod} have curvatures which are solutions to the similar quadratic form $$\mathcal{S}(X_1,\ldots,X_{n+2})=\left(\sum_{i=1}^{n+2} X_i \right)^2 - n \sum_{i=1}^{n+2} X_i^2.$$ \end{rem} 

The mutually tangent circles $C_1,\ldots,C_{q+1}$ of Lemma \ref{tangentcircles} lie on the line $\mathbb{P}^1(\mathbb{F}_q(T))_\infty$, which is canonically identified with the boundary $\partial \mathcal{T}_{q+1}$ of the Bruhat-Tits tree $\mathcal{T}_{q+1}$. Provided that $q$ is odd, i.e., $p\neq 2$, the transitive action of $G_{2,q}$ on $\mathcal{T}_{q+1}$ corresponds to an action on a component (precisely, the quadratic component, or \emph{tree}) of the Bruhat-Tits building $\mathcal{B}_q$ of $\mathrm{SL}(V_q)$, a simplicial complex of dimension $q-2$ \cite[Chapter II.1]{Ser}. On the tree of the building, one can show a climbing lemma and periodicity of the ``river" for quadratic forms \cite{Wij}. We leave the analogous study of isotropy on $\mathcal{B}_q$ as an open question, which we will address in a future work.

We now relate the unit circle $\mathbb{S}_q$ to the Bruhat-Tits building $\mathcal{B}_q$. We first prove both a normal integral basis theorem and a normal basis theorem, which give natural embeddings of $\mathrm{Gal}(V_q/((\mathbb{F}_q(T))_\infty)$ into $\mathrm{SL}(V_q)$.

\begin{lem}[Normal integral basis] \label{normalintegralbasis} Let $\lambda \in  \Lambda_{q,T}\backslash\{0\} \subset \mathbb{T}_q$. The extension $V_q/((\mathbb{F}_q(T))_\infty$ possesses a normal integral basis $\mathfrak{N}_q(\lambda)$, which is described explicitly in terms of $\lambda$. \end{lem}

\begin{proof} The element $T^{-1} \in (\mathbb{F}_q(T))_\infty$ is prime. Let $\lambda \in \Lambda_{q,T}\backslash\{0\}$. Then $\lambda$ generates the field $V_q$ and satisfies $\lambda^{q-1} = -T$. As $e(\mathfrak{P}|\mathfrak{p}) = q-1$, we obtain $$(q-1)v_\mathfrak{P}(\lambda) = v_\mathfrak{P}(\lambda^{q-1}) = v_\mathfrak{P}(-T) = v_\mathfrak{P}(T)  = (q-1)v_{\mathfrak{p}_\infty}(T) = q-1.$$ Hence $v_\mathfrak{P}(\lambda) = 1$, i.e., $\lambda^{-1}$ is a prime element in $V_q$. Let $\theta = \sum_{i=0}^{q-2} \lambda^-i$, and let $\zeta \in \mathbb{F}_q^*$ be a primitive $(q-1)$st root of unity. The set \begin{equation}\label{powerbasis} P_\lambda = \{1,\lambda^{-1},\ldots,\lambda^{-(q-2)}\}\end{equation} clearly forms an integral basis of $V_q/(\mathbb{F}_q(T))_\infty$. The conjugates $\theta,\sigma(\theta),\ldots,$ $\sigma^{q-2}(\theta)$ are generated from the power basis $P_\lambda$ \eqref{powerbasis} via the square matrix $M$ with entries equal to $\zeta^{ij}$ with $i,j \in \{0,\ldots,q-2\}$. By definition, the valuation of $\det(M)$ is equal to zero at the place above $\mathfrak{p}_\infty$, so that the conjugates of $\theta$ form a normal integral basis $\mathfrak{N}_q(\lambda)$ of $V_q/(\mathbb{F}_q(T))_\infty$ as desired. \end{proof}

Of course, in general, elements of $\mathbb{S}_q$ are not prime at $\mathfrak{P}|\mathfrak{p}_\infty$. Still, we easily also obtain a canonical normal basis from each element of $\mathbb{T}_q$, as well as for certain points in $\mathbb{S}_q\backslash \mathbb{T}_q$:

\begin{cor}[Normal basis]\label{normalbasis} Suppose that $\lambda \in \mathbb{S}_q$ is a primitive element for $V_q/(\mathbb{F}_q(T))_\infty$ (in particular, this is satisfied by any $\lambda \in \mathbb{T}_q\backslash\{0\}$). Then the extension $V_q/((\mathbb{F}_q(T))_\infty$ possesses a normal basis $\mathfrak{N}_{q}(\lambda)$, which is described explicitly in terms of $\lambda$. \end{cor}

\begin{proof} By Lemma \ref{Kummerconverse}, if $\lambda \in \Lambda_{q,A}\backslash\{0\}$ ($A \in \mathbb{F}_q[T]$), then there exists $M \in (\mathbb{F}_q(T))_\infty$ such that $\lambda^{q-1} = M$, and by Lemma \ref{Pninf}, $\lambda$ is a primitive element of $V_q/((\mathbb{F}_q(T))_\infty$. Hence the same argument as that of Lemma \ref{normalintegralbasis} shows that the conjugates of $\theta = \sum_{i=0}^{q-2} \lambda^{-i}$ form a normal basis of $V_q/((\mathbb{F}_q(T))_\infty$. The result holds similarly for any $\lambda\in \mathbb{S}_q$, provided that $\lambda$ is of degree $q-1$ over $((\mathbb{F}_q(T))_\infty$. \end{proof}

Clearly, the normal bases $\mathfrak{N}_{q}(\lambda)$ of Lemma \ref{normalintegralbasis} and \ref{normalbasis} coincide when $\lambda \in \Lambda_{q,T} \backslash\{0\}$. We may thus summarise the embedding theorem as:

  \begin{thm}\label{embedding} Let $\lambda \in \mathbb{S}_q$. For the normal basis $\mathfrak{N}_{q}(\lambda)$ of Corollary \ref{normalbasis}, there is a natural embedding $$\iota(\lambda): \mathrm{Gal}(V_q/	(\mathbb{F}_q(T))_\infty) \xhookrightarrow{} \mathrm{SL}(V_q) =\mathrm{SL}_{q-1}((\mathbb{F}_q(T))_\infty).$$
\end{thm}

\begin{proof} This follows immediately from Lemma \ref{normalintegralbasis}, as $\mathrm{Gal}(V_q/	(\mathbb{F}_q(T))_\infty)$ acts as a permutation matrix on the normal basis $\mathfrak{N}_q$.	
\end{proof}

By the embedding of Corollary \ref{embedding}, $\mathrm{Gal}(V_q/	(\mathbb{F}_q(T))_\infty)$ acts in a natural way on the Bruhat-Tits building $\mathcal{B}_q$ via its action on the normal integral basis $\mathfrak{N}_q(\lambda)$. 

\subsection{The hyperbolic plane}

We are now able to define the hyperbolic plane.

\begin{defn}[Hyperbolic plane]\label{def:hyperbolic} $$\mathfrak{H}_q = V_q \backslash (\mathbb{F}_q(T))_\infty.$$
\end{defn}

Thus, by Theorem \ref{vsq-1}, the hyperbolic plane $\mathfrak{H}_q$ is of finite dimension over the real line $(\mathbb{F}_q(T))_\infty$, just as it is for the classical upper half-plane in $\mathbb{C}/\mathbb{R}$. As a rigid analytic space, $\mathfrak{H}_q$ is clearly geometrically connected \cite[Definition 8.1]{Goss3}. Definition \ref{def:hyperbolic} is consistent with the construction $\mathfrak{H} = \overline{(\mathbb{F}_q(T))_\infty}\backslash (\mathbb{F}_q(T))_\infty$ of Drinfel'd and Goss \cite{Dri,Goss2}. Modular forms are defined in terms of the action on the hyperbolic plane $\mathfrak{H}_q$ and unit circle $\mathbb{S}_q$ via the M\"{o}bius transformations of Definitions \ref{def:mobiusadditive} and \ref{def:mobius}.

We recall for each $n \in \mathbb{N}$ that $[n]=T^{q^n} - T$ and $D_n = [1]^{q^n} \cdots [n-1]^q [n]$. Set $D_0 = 1$. The exponential function and analogue of $2\pi i$ for the Carlitz action are given by \cite{Goss2} $$e(z) = \sum_{n=0}^\infty (-1)^n \frac{z^{q^n}}{D_n},\qquad \tilde{\pi} = \prod_{n=1}^\infty \left( 1 - \frac{[n]}{[n+1]}\right).$$ We state the following result for completeness, which is the analogue of the statement that $2\pi i \in \mathbb{C}$.

\begin{lem}$\tilde{\pi} \in V_q$.
\end{lem}
\begin{proof} By definition of $[n]$, the partial products $\tilde{\pi}_N = \prod_{n=1}^N \left( 1 - [n]/[n+1]\right)$ lie in $\mathbb{F}_q(T)$. One may easily show that $\{\tilde{\pi}_N\}_{N \in \mathbb{N}}$ is convergent at $\mathfrak{p}_\infty$, whence $\lim_{N \rightarrow \infty} \tilde{\pi}_N = \tilde{\pi} \in (\mathbb{F}_q(T))_\infty \subset V_q$.
\end{proof}

The natural M\"{o}bius action of $\mathrm{GL}(V_q)=\mathrm{GL}_{q-1}((\mathbb{F}_q(T))_\infty)$ on $\mathfrak{H}_q$ from the left (Definition \ref{def:mobius}) induces an action of the discrete subgroup $\mathrm{SL}(V_q) = \mathrm{SL}_{q-1}((\mathbb{F}_q(T))_\infty)$ on $\mathfrak{H}_q$. We let $L,L'$ be lattices in the Bruhat-Tits building $\mathcal{B}_q$ of $\mathrm{SL}(V_q)$ which belong to the classes $\Lambda,\Lambda'$. By definition, the distance function for two such classes of lattices $\Lambda' \subset \Lambda$ is equal to the positive integer $n$ for which $$L/L' \cong \vartheta_\mathfrak{P}/ \lambda^n \vartheta_\mathfrak{P},$$ where $\lambda \in \Lambda_{q,T} \backslash\{0\}$ is a prime element of $V_q$ \cite[\S 1.1.1]{Ser}. This agrees with the metric on the homogeneous analogue of Drinfel'd \cite[\S 6]{Dri} for the distance function on $(\mathbb{F}_q(T))_\infty^{\otimes d}$ when $d=q-1$.

We now give the definition of a modular form. In the following definition, we use the phrase \emph{automorphic form} to mean that a function satisfies certain transformation rules according to a fixed weight. 

\begin{defn}[Modular forms for $\mathfrak{H}_q$] \label{def:Modular} A \emph{modular form} is a rigid analytic function defined on $\mathfrak{H}_q$ which is automorphic for $\mathrm{SL}(V_q)$. \end{defn}

As noted by Goss \cite[\S 4]{Goss2}, the weight must be divisible by $q-1 = \dim(V_q/(\mathbb{F}_q(T))_\infty)$, the dimension of the hyperbolic plane $\mathfrak{H}_q$ over the real line, in analogy to the statement that the weight of a classical modular form must be even, i.e., divisible by the dimension of the upper half-plane $\mathbb{H}$ over $\mathbb{R}$. We may now define the Eisenstein series.

\begin{defn}[Eisenstein series for $\mathfrak{H}_q$]\label{def:Eisenstein} Let $L \subset V_q$ be a lattice, and let $k \in \mathbb{N}$. The \emph{Eisenstein series} $E^{(q-1)k}(L)$ is defined as \begin{equation}\label{Eisenstein} E^{(q-1)k}(L) = \sum_{0 \neq \alpha \in L} \alpha^{-(q-1)k}.\end{equation} \end{defn}

\begin{rem}
	We note that Definitions \ref{def:Modular} and \ref{def:Eisenstein} are not the same as those given in \cite{Goss1,Goss2}, which express automorphic forms in terms of linear fractional transformations.
\end{rem}

By \cite[Theorem 2.19]{Goss3}, it is not hard to show that the series $E^{(q-1)k}(L)$ are holomorphic at the cusps. The lattice sum expression \eqref{Eisenstein} is particularly convenient, as one may consider the Eisenstein series in terms of the $\vartheta_\mathfrak{P}$-lattices of $V_q$ in the Bruhat-Tits building. These modular forms and Eisenstein series will be studied in a future work.

\section{Concluding remarks}
We have restricted our attention in this analysis to rank one Drinfel'd modules \cite{Hay2}, in order to highlight the vast similarities between the $q$-unit circle $\mathbb{S}_q$ and the classical unit circle $S^1$. In addition to the other open questions posed here, we invite the interested reader to consider the constructions appearing in this work for the Drinfel'd modules of higher rank \cite{Hay3}. Is there a similar analogy to $S^n$ ($n > 1$)? We leave this, too, as a fascinating open question.

The author sincerely thanks J. Lansky and A. Kontogeorgis for their unwavering encouragement and many thoughtful discussions.

\Addresses
\end{document}